\documentclass[10pt]{amsart}

\usepackage[letterpaper,textwidth=6in,textheight=8in,centering]{geometry}
\usepackage{amsmath,amssymb,amsthm,mathtools}
\usepackage{booktabs}
\usepackage{graphicx}
\usepackage{subcaption}
\usepackage{enumitem}
\usepackage{microtype}
\usepackage{hyperref}
\usepackage{xcolor}
\usepackage{placeins}
\usepackage{float}
\graphicspath{{figures/}}

\hypersetup{
  colorlinks=true,
  linkcolor=blue,
  citecolor=blue,
  urlcolor=blue
}

\newtheorem{theorem}{Theorem}[section]
\newtheorem{proposition}[theorem]{Proposition}
\newtheorem{definition}[theorem]{Definition}

\newcommand{\R}{\mathbb R}
\newcommand{\N}{\mathbb N}
\newcommand{\E}{\mathbb E}
\newcommand{\J}{\mathcal J}
\newcommand{\Dcal}{\mathcal D}
\newcommand{\Acal}{\mathcal A}
\newcommand{\Hcal}{\mathcal H}
\newcommand{\Kcal}{\mathcal K}
\newcommand{\Lcal}{\mathcal L}
\newcommand{\Pcal}{\mathcal P}
\newcommand{\Qcal}{\mathcal Q}
\newcommand{\Rcal}{\mathcal R}
\newcommand{\Scal}{\mathcal S}
\newcommand{\Zcal}{\mathcal Z}
\newcommand{\TV}{\operatorname{TV}}
\newcommand{\Hell}{\operatorname{Hell}}

\title[Resolution-stable sampling under refinement]{Posterior Convergence without Force Convergence: Resolution-Stable Sampling for Rough Bayesian Inverse Problems}

\author{Zhiliang Deng$^*$}
\thanks{$*$ School of Mathematical Sciences, University of Electronic Science and Technology of China, Chengdu, China. Email: dengzhl@uestc.edu.cn}

\author{Xiaomei Yang$^\dag$}
\thanks{$\dag$ School of Mathematics, Southwest Jiaotong University, Chengdu, China. Email: yangxiaomath@swjtu.edu.cn}


\begin{document}
\raggedbottom

\begin{abstract}
Bayesian targets may converge under model refinement even when the exact sensitivities used by gradient-based samplers do not. We study this probability--sensitivity mismatch and its consequences for Metropolized Hamiltonian proposals. A vanishing-amplitude wiggly-energy model first gives the basic analytic obstruction: the potential perturbation tends to zero while its classical derivative is of order $r_\varepsilon/\varepsilon$. We then show that the same scaling arises naturally in a periodic elliptic inverse problem, where homogenization makes the forward map and Gaussian likelihood converge while differentiation with respect to a microscopic scale parameter retains an $O(1)$ oscillatory contribution. This provides a PDE origin for the single-scale wiggly mechanism. The main construction concerns a more demanding nested Weierstrass hierarchy, interpreted as an analytically tractable prototype for repeated corrector contributions across geometrically separated scales. There all previously resolved scales persist, adjacent classical-force increments grow geometrically like $(ab)^N$, and the limiting rough component may fail to possess a classical derivative. In this self-similar setting the matched Jackson quotient is structurally adapted to the refinement through dilation covariance and exact finite closure. Uniform negative-log-likelihood approximation yields explicit total-variation, Hellinger, and bounded quantity-of-interest bounds. Measurable kick--drift--kick maps remain exact after Metropolis correction, local field convergence propagates to fixed-length proposals and kernels, and the first classical HMC half-kick can have no fixed-step refinement limit. Numerical experiments on scale-structured inverse problems test the resulting resolution-stability mechanism across one- and two-dimensional inverse problems.
\end{abstract}

\maketitle

\section{Introduction}
\label{sec:introduction}

Hamiltonian Monte Carlo (HMC) is naturally formulated for differentiable
potentials: the potential defines a force, a reversible numerical integrator
approximates the Hamiltonian flow, and a Metropolis correction removes the
remaining time-discretization bias. This setting underlies results on tuning,
geometric integration, ergodicity, and dimension-dependent scaling; see \cite{BeskosEtAl2013, BouRabeeSanzSerna2018, DuaneEtAl1987, DurmusMoulinesSaksman2020, HairerLubichWanner2006, Neal2011}. Nonsmooth variants address
convex composite energies, truncation boundaries, finitely many kinks, and
discontinuity interfaces \cite{ChaariEtAl2016,PakmanPaninski2014}. A different
difficulty arises under model refinement: the target may stabilize while the
newly resolved small scales are amplified by differentiation.

A concrete precursor to this difficulty comes from materials mechanics.
Abeyaratne et al. \cite{AbeyaratneChuJames1996} introduced
wiggly-energy models to describe hysteresis and twin evolution observed in
Cu--Al--Ni shape-memory alloys. Their model superposes many small energy
corrugations on a slowly varying background, with the microscopic picture
motivated by repeated tip-splitting events in martensitic microstructures.
Menon \cite{Menon2002} subsequently studied the associated averaging problem
for energies of the form
$$
F(x)+\varepsilon A(x/\varepsilon),
$$
and showed explicitly that convergence at the energy level need not be
accompanied by strong convergence of the gradients. Thus the separation
between a small energy perturbation and a non-small force has both a physical
origin and a precise mathematical formulation.

The same separation appears in related analytical and computational settings.
In variational analysis, Sandier and Serfaty
\cite{SandierSerfaty2004} developed conditions controlling the passage from
energy convergence to convergence of the associated gradient flows.
Wiggly-gradient systems subsequently motivated
energy--dissipation-principle convergence and its relaxed variants
\cite{DondlFrenzelMielke2019}. In MCMC, Plech\'a\v{c} and Simpson
\cite{PlechacSimpson2020} showed that fine-scale roughness can substantially
degrade gradient-based Langevin proposals, while Livingstone and Zanella
\cite{LivingstoneZanella2022} analyzed the loss of robustness of
gradient-based methods under scale heterogeneity. Approximate-gradient HMC
can nevertheless retain the exact target after a suitable Metropolis
correction \cite{LiEtAl2019}. These results clarify the effects of roughness
on dynamics and sampling, but they do not address the refinement question
considered here: whether convergence of a sequence of Bayesian targets is
accompanied by convergence of the sensitivities used to construct
gradient-based proposals.

We study this question first through the elementary family
$$
U_\varepsilon(x)
=
V(x)+r_\varepsilon\psi(x/\varepsilon),
\qquad
r_\varepsilon\to0.
$$
The perturbation of the potential is $O(r_\varepsilon)$, whereas the
derivative of the oscillatory component is of order
$r_\varepsilon/\varepsilon$. Hence the target perturbation may vanish while
the corresponding force remains of order one, or even grows. This elementary
scaling isolates the probability--sensitivity mismatch that will later be
sharpened for the nested Weierstrass hierarchy.

The same mechanism also arises from PDE refinement rather than from an
explicitly prescribed oscillatory potential. In periodic elliptic
homogenization, the state contains small-scale correctors of order
$O(\varepsilon)$, while differentiation with respect to a parameter entering
the fast variable introduces the inverse microscopic scale. Bayesian inverse
homogenization already provides settings in which fine-scale and effective
posteriors converge while microscopic correctors remain relevant
\cite{AbdulleDiBlasio2020, HoangQuek2019}. In Section~\ref{subsec:pde-origin} we consider a one-dimensional periodic-conductivity
inverse problem for which the forward and negative-log-likelihood errors decay approximately linearly in $\varepsilon$, whereas the exact likelihood
sensitivity remains at an $O(1)$ scale. In this example the oscillation is generated by the PDE coefficient itself rather than inserted directly into the posterior.

The main construction is then developed for a more demanding nested Weierstrass hierarchy. At resolution $N$, all coarser oscillations remain present while a new scale is added. The potential tail decays like $a^N$, but the adjacent classical-force increment grows like $(ab)^N$ in the rough regime $ab>1$. Hence the difficulty is more persistent than in a single-scale wiggly perturbation: the refinement accumulates a hierarchy of unresolved scales and the limiting rough component can be genuinely nondifferentiable. This multiplicative self-similarity also supplies additional structure. The Jackson quotient \cite{Jackson1908}, taken at the intrinsic dilation ratio $b$, is covariant under dilation and admits an exact finite closure under the Weierstrass scale recursion \cite{ErzanEckmann1997,YangDeng2025,YangDeng2026}. These properties identify the intrinsic dilation ratio used by the scale-matched $q$-difference proposal studied below; additive and mismatched multiplicative quotients are retained as comparison fields.

The paper makes four contributions. First, it quantifies the probability--force separation for vanishing-amplitude wiggly perturbations and then sharpens it to a geometric refinement law for the nested Weierstrass family. Second, it derives total-variation, Hellinger, and bounded quantity-of-interest bounds from uniform negative-log-likelihood approximation and gives a Gaussian forward-map criterion for Bayesian inverse problems. Third, it proves exactness of Metropolized kick--drift--kick proposals driven by measurable fields and propagates local field convergence to fixed-length proposals, acceptance functions, and Markov kernels on common compact sets. Fourth, it identifies the loss of classical HMC consistency already at the first half-kick and derives the necessary refinement-dependent step-size scale. The matched Jackson construction is a structurally adapted finite-scale field for multiplicatively self-similar refinement.

The remainder of the paper is organized as follows. Section~\ref{sec:rough-structure} begins with the wiggly-energy obstruction and its periodic-PDE origin, then develops the nested Weierstrass hierarchy, the matched Jackson field, and posterior convergence. Section~\ref{sec:resolution-kernels} establishes exactness and refinement convergence of Metropolized finite-scale proposals and contrasts them with the classical half-kick. Section~\ref{sec:algorithms} presents the scale-structured sampling experiments. Section~\ref{sec:conclusion} summarizes the conclusions and limitations. The uniform H\"older estimate used in the proposal analysis is proved in the appendix.

\section{Rough target refinement and finite-scale proposal fields}
\label{sec:rough-structure}

This section separates probability-level convergence from force-level refinement. We first use a single-scale wiggly perturbation to expose the obstruction in its simplest analytic form and then show how the same scaling is generated by a standard periodic elliptic inverse problem. This provides a bridge from homogenization correctors to rough posterior sensitivities. We then pass to the nested Weierstrass hierarchy, viewed as a canonical model for repeated corrector contributions across geometrically separated scales. Its intrinsic multiplicative recursion leads naturally to the matched Jackson field. The section closes with regularity and posterior-convergence estimates used later in the sampling analysis.

\subsection{A motivating wiggly-energy obstruction}
\label{subsec:wiggly-generalization}

Consider a bounded one-periodic profile $\psi$ and microscopic lengths $\varepsilon_n\downarrow0$. Let $r_n>0$ satisfy $r_n\to0$, and define
$$
R_n(x)=r_n\psi\left(\frac{x}{\varepsilon_n}\right),
\qquad
U_n(x)=V(x)+R_n(x).
$$
At the level of the potential, $R_n$ is small because its amplitude is $r_n$. Exact differentiation sees a different scale:
$$
R_n'(x)=\frac{r_n}{\varepsilon_n}\psi'\left(\frac{x}{\varepsilon_n}\right).
$$
The microscopic wavelength therefore enters with the inverse factor $\varepsilon_n^{-1}$.

\begin{proposition}[Vanishing-amplitude probability--force separation]
\label{prop:wiggly-separation}
Let $\psi\in H^1_{\mathrm{per}}(0,1)\cap L^\infty(0,1)$ with $\|\psi'\|_{L^2(0,1)}>0$. Suppose $\varepsilon_n^{-1}\in\N$ and $r_n\to0$. Then
\begin{equation}
\|R_n\|_{L^\infty(\R)}\leq
r_n\|\psi\|_{L^\infty},
\label{eq:wiggly-potential-scale}
\end{equation}
whereas
\begin{equation}
\|R_n'\|_{L^2(0,1)}=
\frac{r_n}{\varepsilon_n}
\|\psi'\|_{L^2(0,1)}.
\label{eq:wiggly-force-scale}
\end{equation}
Consequently, the potential perturbation vanishes while the exact rough-force contribution remains $O(1)$ when $r_n\asymp\varepsilon_n$ and diverges whenever $r_n/\varepsilon_n\to\infty$. In particular, for $r_n=\varepsilon_n^\alpha$ with $0<\alpha<1$,
$$
\|R_n\|_{L^\infty}=O(\varepsilon_n^\alpha),
\qquad
\|R_n'\|_{L^2(0,1)}\asymp\varepsilon_n^{\alpha-1}\longrightarrow\infty.
$$
\end{proposition}

\begin{proof}
The first bound is immediate. Since
$$
R_n'(x)=\frac{r_n}{\varepsilon_n}\psi'\left(\frac{x}{\varepsilon_n}\right),
$$
periodicity and $\varepsilon_n^{-1}\in\N$ give
$$
\int_0^1\left|\psi'\left(\frac{x}{\varepsilon_n}\right)\right|^2dx=\|\psi'\|_{L^2(0,1)}^2,
$$
which proves equation~\eqref{eq:wiggly-force-scale}.
\end{proof}

\subsection{A periodic elliptic inverse problem generating the wiggly mechanism}
\label{subsec:pde-origin}

The preceding scaling also arises from a standard multiscale PDE. Periodic and
locally periodic coefficients are classical models for composite and porous media,
and Bayesian inverse homogenization considers the recovery of microscopic
coefficients or macroscopic parameterizations from observations of fine-scale
solutions \cite{AbdulleDiBlasio2020,HoangQuek2019}. The point relevant here is
that an $O(\varepsilon)$ corrector in the state need not produce an
$O(\varepsilon)$ parameter sensitivity when the unknown enters the fast phase.

Consider
\begin{equation}
-\frac{d}{dx}
\left[
k\left(\frac{\theta x}{\varepsilon}\right)
\frac{dp_\varepsilon}{dx}(x;\theta)
\right]
=1,
\qquad
p_\varepsilon(0;\theta)=p_\varepsilon(1;\theta)=0,
\label{eq:periodic-pde}
\end{equation}
where $k$ is positive, one-periodic, and $C^1$, and
$\theta\in[\theta_-,\theta_+]\subset(0,\infty)$ calibrates the microscopic
wavelength $\varepsilon/\theta$. Set
$$
m(y)=\frac{1}{k(y)},
\qquad
\overline m=\int_0^1 m(y)\,dy,
\qquad
k_{\rm hom}=\overline m^{-1}.
$$

The one-dimensional structure gives an exact representation of the solution.
Integrating equation~\eqref{eq:periodic-pde} once yields
$$
k\left(\frac{\theta x}{\varepsilon}\right)
p_\varepsilon'(x;\theta)
=
C_\varepsilon(\theta)-x.
$$
Hence
$$
p_\varepsilon'(x;\theta)
=
\left[C_\varepsilon(\theta)-x\right]
m\left(\frac{\theta x}{\varepsilon}\right).
$$
Integrating from $0$ to $x$ and using $p_\varepsilon(0;\theta)=0$ gives
\begin{equation}
p_\varepsilon(x;\theta)
=
C_\varepsilon(\theta)I_\varepsilon(x;\theta)
-
J_\varepsilon(x;\theta),
\label{eq:periodic-flux-representation}
\end{equation}
where
\begin{align*}
I_\varepsilon(x;\theta)
&=
\int_0^x
m\left(\frac{\theta s}{\varepsilon}\right)\,ds,
\\
J_\varepsilon(x;\theta)
&=
\int_0^x
s\,m\left(\frac{\theta s}{\varepsilon}\right)\,ds.
\end{align*}
The second boundary condition gives
$$
C_\varepsilon(\theta)
=
\frac{J_\varepsilon(1;\theta)}
{I_\varepsilon(1;\theta)}.
$$
Since $k>0$, one has $I_\varepsilon(1;\theta)>0$, so the representation is
well defined.

To identify the homogenization scale, choose the one-periodic corrector
$\chi$ with zero mean such that
$$
\chi'(y)=m(y)-\overline m,
\qquad
\int_0^1\chi(y)\,dy=0.
$$
Then
\begin{equation}
I_\varepsilon(x;\theta)
=
\overline m x
+
\frac{\varepsilon}{\theta}
\left[
\chi\left(\frac{\theta x}{\varepsilon}\right)-\chi(0)
\right].
\label{eq:periodic-corrector-identity}
\end{equation}
Since $\chi$ has zero mean, let $\Xi$ be a one-periodic primitive satisfying
$\Xi'=\chi$. Integration by parts gives
\begin{equation}
J_\varepsilon(x;\theta)
=
\frac{\overline m}{2}x^2
+
\frac{\varepsilon x}{\theta}
\chi\left(\frac{\theta x}{\varepsilon}\right)
-
\left(\frac{\varepsilon}{\theta}\right)^2
\left[
\Xi\left(\frac{\theta x}{\varepsilon}\right)-\Xi(0)
\right].
\label{eq:periodic-J-corrector}
\end{equation}
In particular,
\begin{align*}
I_\varepsilon(1;\theta)=
\overline m+\frac{\varepsilon}{\theta}
\left[\chi\left(\frac{\theta}{\varepsilon}\right)-\chi(0)\right],
\quad
J_\varepsilon(1;\theta)=\frac{\overline m}{2}+
\frac{\varepsilon}{\theta}\chi\left(\frac{\theta}{\varepsilon}\right)+O(\varepsilon^2),
\end{align*}
uniformly for $\theta\in[\theta_-,\theta_+]$. Expanding the quotient gives
\begin{equation}
C_\varepsilon(\theta)=\frac12+
\frac{\varepsilon}{2\theta\overline m}
\left[\chi\left(\frac{\theta}{\varepsilon}\right)+\chi(0)\right]+O(\varepsilon^2).
\label{eq:C-eps-expansion}
\end{equation}
Consequently,
$$
p_\varepsilon(x;\theta)=p_0(x)+O(\varepsilon)
$$
uniformly in $x$ and $\theta$, where
\begin{equation}
p_0(x)=\frac{\overline m}{2}x(1-x)=\frac{x(1-x)}{2k_{\rm hom}}.
\label{eq:periodic-homogenized-solution}
\end{equation}

The parameter sensitivity has a different scale. Writing $Y=\frac{\theta x}{\varepsilon}$,
the integral quantities satisfy the exact identities
\begin{equation}
\partial_\theta I_\varepsilon(x;\theta)=
\frac{x\,m(Y)-I_\varepsilon(x;\theta)}{\theta},
\qquad
\partial_\theta J_\varepsilon(x;\theta)=
\frac{x^2m(Y)-2J_\varepsilon(x;\theta)}{\theta}.
\label{eq:periodic-IJ-sensitivity}
\end{equation}
Therefore
\begin{equation}
\partial_\theta C_\varepsilon(\theta)=
\frac{\partial_\theta J_\varepsilon(1;\theta)I_\varepsilon(1;\theta)-J_\varepsilon(1;\theta)\partial_\theta I_\varepsilon(1;\theta)}{I_\varepsilon(1;\theta)^2},
\label{eq:periodic-C-sensitivity}
\end{equation}
and
\begin{equation}
\partial_\theta p_\varepsilon(x;\theta)
=
\partial_\theta C_\varepsilon(\theta)I_\varepsilon(x;\theta)
+
C_\varepsilon(\theta)\partial_\theta I_\varepsilon(x;\theta)
-
\partial_\theta J_\varepsilon(x;\theta).
\label{eq:periodic-p-sensitivity}
\end{equation}
The first identity in equation~\eqref{eq:periodic-IJ-sensitivity} is equivalent to
\begin{equation}
\partial_\theta I_\varepsilon(x;\theta)
=
-
\frac{\varepsilon}{\theta^2}
\left[
\chi\left(\frac{\theta x}{\varepsilon}\right)-\chi(0)
\right]
+
\frac{x}{\theta}
\left[
m\left(\frac{\theta x}{\varepsilon}\right)-\overline m
\right].
\label{eq:periodic-sensitivity-identity}
\end{equation}
The first term is $O(\varepsilon)$, whereas the second remains oscillatory at
$O(1)$. Thus the $O(\varepsilon)$ state corrector does not imply an
$O(\varepsilon)$ parameter sensitivity.

We now quantify this separation. Let $k(y)=\exp\left[0.8\sin(2\pi y)\right]$,
$\theta\in[0.9,1.1]$,
and observe the solution at $x_j\in\{0.17,0.31,0.46,0.64,0.83\}$.
Define $\mathcal G_\varepsilon(\theta)=\left(p_\varepsilon(x_1;\theta),\ldots,p_\varepsilon(x_5;\theta)\right)^\top$
and $\mathcal G_0=\left(p_0(x_1),\ldots,p_0(x_5)\right)^\top$.
The synthetic data are fixed as
$$
y=\mathcal G_0+\sigma
\begin{pmatrix}
0.45 & -0.30 & 0.60 & -0.55 & 0.25
\end{pmatrix}^{\!\top},
\qquad
\sigma=0.03.
$$
For inference we adopt the Gaussian observation model with covariance $\sigma^2 I$, giving the negative log-likelihood
The Gaussian negative log-likelihood is
\begin{equation}
\Lcal_\varepsilon(\theta)=
\frac{1}{2\sigma^2}
\left\|
\mathcal G_\varepsilon(\theta)-y
\right\|_2^2,
\label{eq:periodic-likelihood}
\end{equation}
with
$$
\Lcal_0=\frac{1}{2\sigma^2}
\left\|\mathcal G_0-y\right\|_2^2.
$$
Here $\mathcal G_0$ is independent of $\theta$, and hence $\Lcal_0$ is
constant. The exact likelihood sensitivity is
\begin{equation}
\Lcal_\varepsilon'(\theta)=\frac{1}{\sigma^2}
\left[\mathcal G_\varepsilon(\theta)-y\right]^\top
\partial_\theta\mathcal G_\varepsilon(\theta),
\label{eq:periodic-likelihood-gradient}
\end{equation}
where $\partial_\theta\mathcal G_\varepsilon$ is evaluated from
equation~\eqref{eq:periodic-p-sensitivity}, without finite differencing in
$\theta$.

For each $\varepsilon$, define
\begin{align*}
\begin{aligned}
&E_{\rm fwd}(\varepsilon)=\sup_{\theta\in[0.9,1.1]}
\left\|\mathcal G_\varepsilon(\theta)-\mathcal G_0\right\|_2,
&&E_{\Lcal}(\varepsilon)=\sup_{\theta\in[0.9,1.1]}
\left|\Lcal_\varepsilon(\theta)-\Lcal_0\right|,
\\
&S_{\rm rms}(\varepsilon)=\left[\frac{1}{0.2}\int_{0.9}^{1.1}\left|
\Lcal_\varepsilon'(\theta)\right|^2\,d\theta\right]^{1/2},
&&S_{\max}(\varepsilon)=\sup_{\theta\in[0.9,1.1]}
\left|\Lcal_\varepsilon'(\theta)\right|.
\end{aligned}
\end{align*}
The first two quantities measure convergence at the forward and
negative-log-likelihood levels. The latter two measure the root-mean-square
and maximum magnitudes of the likelihood sensitivity over the parameter
interval.

We use $\varepsilon=2^{-4},\ldots,2^{-9}$ and evaluate
equations~\eqref{eq:periodic-flux-representation} and
\eqref{eq:periodic-p-sensitivity} on a grid of $101$ equally spaced values of
$\theta$.

\begin{figure}[H]
\centering
\begin{subfigure}[t]{0.49\textwidth}
\centering
\includegraphics[width=\textwidth]{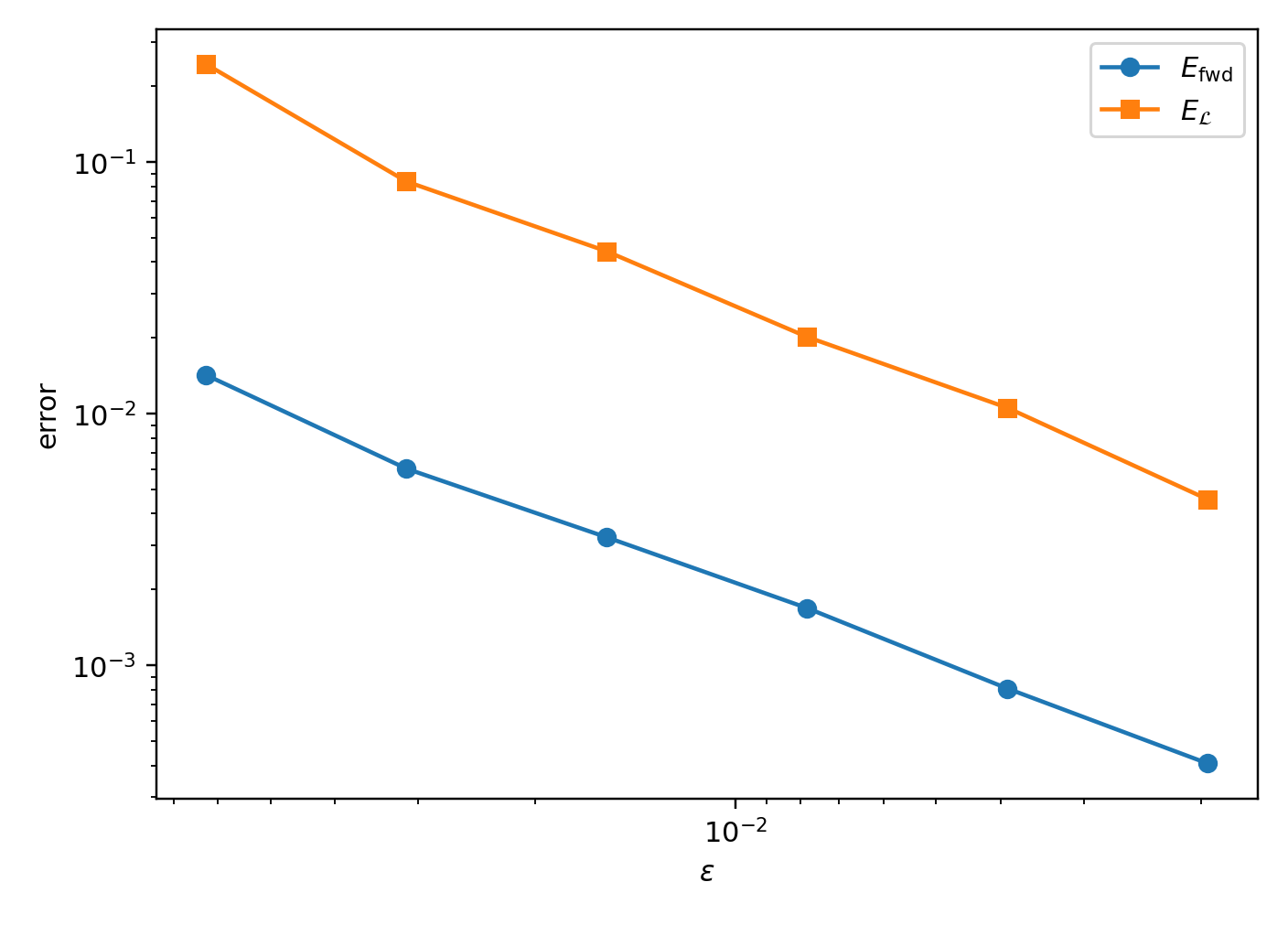}
\caption{Forward and negative-log-likelihood errors.}
\label{fig:periodic-pde-convergence}
\end{subfigure}
\hfill
\begin{subfigure}[t]{0.49\textwidth}
\centering
\includegraphics[width=\textwidth]{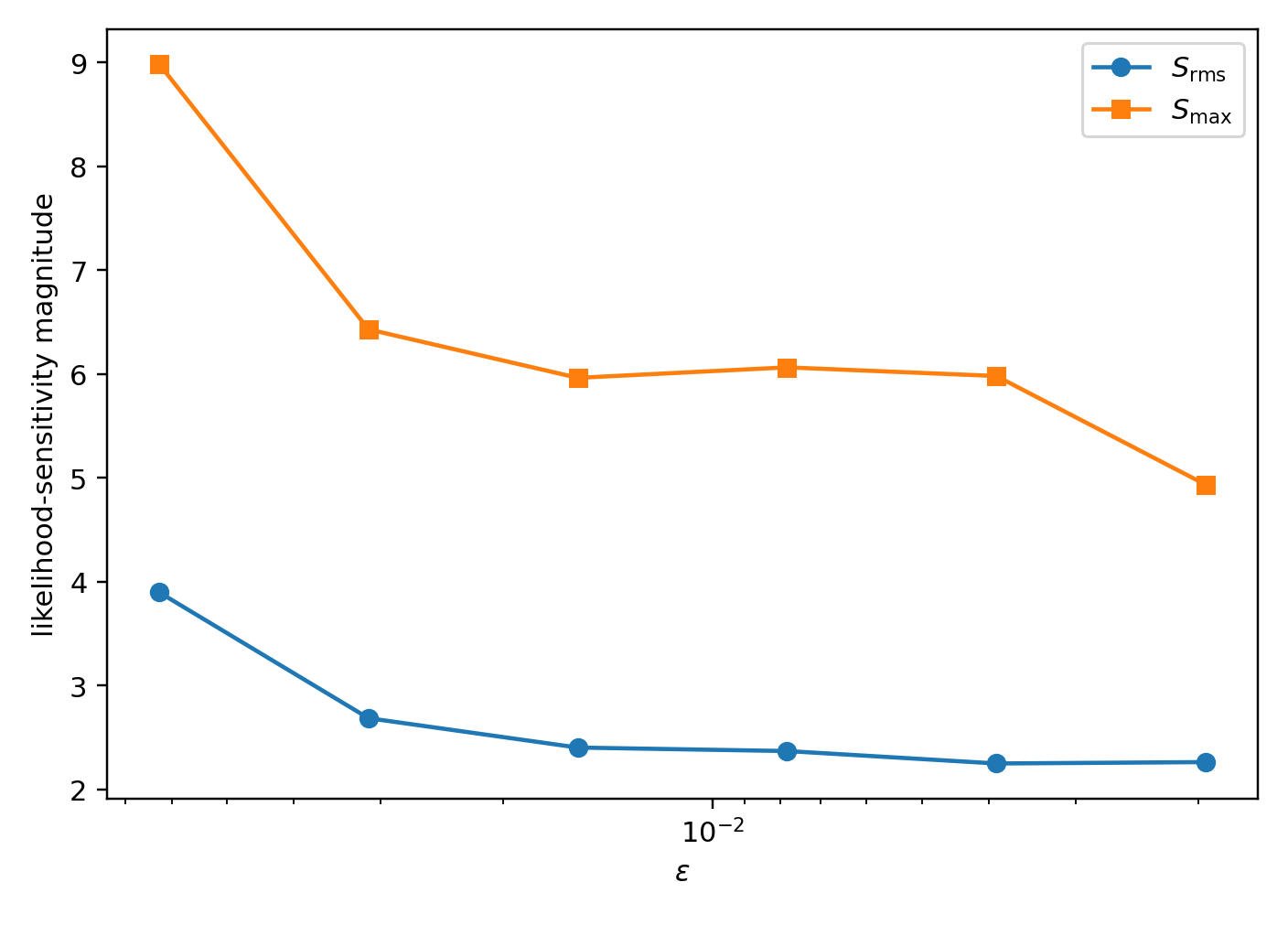}
\caption{Likelihood-sensitivity magnitudes.}
\label{fig:periodic-pde-sensitivity}
\end{subfigure}
\caption{
Refinement diagnostics for the periodic elliptic inverse problem.
Left: $E_{\rm fwd}$ and $E_{\Lcal}$.
Right: $S_{\rm rms}$ and $S_{\max}$.
}
\label{fig:periodic-pde-diagnostic}
\end{figure}

Figure~\ref{fig:periodic-pde-diagnostic} summarizes the four refinement
diagnostics. In Figure~\ref{fig:periodic-pde-convergence}, both
$E_{\rm fwd}$ and $E_{\Lcal}$ decrease steadily as $\varepsilon$ is refined.
By contrast, Figure~\ref{fig:periodic-pde-sensitivity} shows that
$S_{\rm rms}$ and $S_{\max}$ remain at an $O(1)$ scale.

Over the four finest levels, log--log regression gives slopes $1.002$ for
$E_{\rm fwd}$ and $1.076$ for $E_{\Lcal}$. The fitted slope of
$S_{\rm rms}$ is $0.033$. From $\varepsilon=2^{-6}$ to $2^{-9}$,
$S_{\rm rms}$ changes only from $2.40$ to $2.26$, while $S_{\max}$ remains
between $4.94$ and $6.07$. At the finest level,
$$
E_{\rm fwd}=4.07\times10^{-4},
\qquad
E_{\Lcal}=4.56\times10^{-3},
\qquad
S_{\rm rms}=2.26.
$$
The forward map and the negative log-likelihood therefore converge at
essentially first order, whereas the exact likelihood sensitivity remains
nonvanishing over the same sequence of microscopic scales.

The periodic example contains one microscopic scale. Real composite and
porous materials may instead contain several separated periodic scales, and
reiterated homogenization is designed for such nested structures
\cite{AllaireBriane1996,NiuShenXu2020}. Hierarchical porous microstructures
spanning multiple physical scales are also used explicitly in material models
\cite{HeWangPindera2019}. A parameter-dependent corrector expansion can have
the schematic form
$$
\mathcal F_N(\theta)
=
\mathcal F_0(\theta)
+
\sum_{n=0}^{N}
r_n\,
\psi_n\left(\theta,\frac{\theta}{\varepsilon_n}\right)
+
\text{higher-order terms}.
$$
If the separated wavelengths and corrector amplitudes follow approximate
geometric laws
$$
\varepsilon_n\asymp b^{-n},
\qquad
r_n\asymp a^n,
$$
and the microscopic profiles repeat or remain comparable across levels, then
freezing the slow dependence leads to the canonical prototype
$$
\sum_{n=0}^{N}a^n\phi(b^n\theta).
$$
The Weierstrass hierarchy is used here as an analytically tractable model of
repeated periodic corrector contributions under geometric scale separation;
an exact Weierstrass representation is not assumed for a general hierarchical
inverse problem. The next subsection studies the resulting refinement
mismatch.

\subsection{Nested Weierstrass refinement and discrete scale recursion}
\label{subsec:rough-target-family}

Let
$$
V:(0,\infty)\longrightarrow\R
$$
be a measurable baseline potential satisfying
$$
Z_V=\int_0^\infty\exp\left[-V(x)\right]\,dx<\infty.$$
Let $\phi:\R\longrightarrow\R$
be bounded and one-periodic. For $0<a<1$, $b>1$,
define the finite-resolution and limiting rough components by
$$
W_N(x)=\sum_{n=0}^{N}a^n\phi\left(b^nx\right),
\qquad
W(x)=\sum_{n=0}^{\infty}a^n\phi\left(b^nx\right).
$$
For a fixed amplitude parameter $\kappa\in\R$,
we study the potentials
\begin{equation}
U_N(x)=V(x)+\kappa W_N(x),
\qquad
U(x)=V(x)+\kappa W(x).
\label{eq:rough-potentials}
\end{equation}
Their normalizing constants are
$$
Z_N
=
\int_0^\infty
\exp\left[-U_N(x)\right]
\,dx,
\qquad
Z
=
\int_0^\infty
\exp\left[-U(x)\right]
\,dx,
$$
and the corresponding target densities are
$$
\pi_N(x)
=
\frac{\exp\left[-U_N(x)\right]}{Z_N},
\qquad
\pi(x)
=
\frac{\exp\left[-U(x)\right]}{Z}.
$$
Thus $Z_V$ is only the normalizing integral of the baseline density; the full rough targets are normalized by $Z_N$ and $Z$, respectively.

Since $\phi$ is bounded and $0<a<1$, the series defining $W$ converges uniformly and
\begin{equation}
\left\|W-W_N\right\|_{L^\infty(\R)}
\leq
\frac{\|\phi\|_{L^\infty}}{1-a}
a^{N+1}.
\label{eq:potential-tail}
\end{equation}
Moreover,
$$
\|W_N\|_{L^\infty(\R)}
\leq
\frac{\|\phi\|_{L^\infty}}{1-a},
\qquad
\|W\|_{L^\infty(\R)}
\leq
\frac{\|\phi\|_{L^\infty}}{1-a}.
$$
Consequently,
$$
M=\frac{|\kappa|\|\phi\|_{L^\infty}}{1-a},
$$
and
\begin{align*}
e^{-M}Z_V\leq Z_N\leq e^MZ_V,
\quad
e^{-M}Z_V\leq Z \leq e^MZ_V.
\end{align*}
Thus all normalizing constants are finite and strictly positive.

The central refinement question is now visible at the level of the sampling problem. Although
$$
U_N\longrightarrow U
$$
uniformly and the target densities converge, the classical force fields associated with $U_N$ need not possess a refinement limit. The following scale recursion identifies the multiplicative finite-scale structure that remains compatible with the rough component.

A direct shift of the summation index gives
$$
W(bx)
=
\frac{1}{a}
\left[
W(x)-\phi(x)
\right].
$$
More generally, for every integer $m\geq1$,
\begin{equation}
W\left(b^mx\right)
=
\frac{1}{a^m}
\left[
W(x)
-
\sum_{j=0}^{m-1}
a^j\phi\left(b^jx\right)
\right].
\label{eq:scale-recursion-m}
\end{equation}
The finite truncations satisfy the corresponding relation with one terminal remainder,
$$
W_N(bx)
=
\frac{1}{a}
\left[
W_N(x)-\phi(x)
\right]
+
a^N\phi\left(b^{N+1}x\right).
$$
These identities single out the dilation ratios $b^m$ used in the matched Jackson construction.

The comparison with the periodic PDE model is now precise.  The single-scale corrector in Section~\ref{subsec:pde-origin} is replaced as $\varepsilon$ decreases, whereas $W_N$ retains every previously resolved scale and adds the new term $a^N\phi(b^Nx)$.  The mismatch is therefore more pronounced in the refinement sense for three reasons: the unresolved scales accumulate, the adjacent classical-force increments grow geometrically while the potential tail decays geometrically, and the limiting $W$ may remain genuinely nondifferentiable.  These facts are quantified in equation~\eqref{eq:potential-tail} and Theorem~\ref{thm:classical-force-growth}.  At the same time, the hierarchy has extra structure: all frequencies are related by the intrinsic dilation $x\mapsto bx$.  It is this discrete multiplicative self-similarity, rather than roughness alone, that motivates the matched Jackson construction.

\subsection{The Jackson quotient as a scale-covariant two-point operator}

This subsection identifies the multiplicative difference structure used later. We first explain the elementary normalization of the Jackson quotient, then record its covariance under dilation, and finally use the Weierstrass recursion to obtain exact closure at matched dilation ratios.

For $c>1$, let
$$
\Scal_c f(x)
=
f(cx)
$$
be the dilation operator. Define
\begin{equation*}
\J_c f(x)=
\frac{f(cx)-f(x)}{(c-1)x},
\qquad x>0.
\end{equation*}
The normalization is canonical within the class of two-point operators. If
$$
\Dcal_cf(x)=A_c(x)f(cx)+B_c(x)f(x)
$$
annihilates constants and satisfies $\Dcal_c x=1$, then
$$
A_c(x)=\frac{1}{(c-1)x},
\qquad
B_c(x)=-\frac{1}{(c-1)x},
$$
so $\Dcal_c=\J_c$. Thus $\J_c$ is the normalized quotient associated with the multiplicative displacement $x\mapsto cx$.

More importantly for the analysis below, the operator is covariant under dilation. For $\Scal_sf(x)=f(sx)$,
one has
\begin{equation}
\J_c\Scal_s=s\Scal_s\J_c.
\label{eq:scale-covariance}
\end{equation}
%Indeed,
%\begin{align*}
%\J_c\Scal_sf(x)=
%\frac{f(scx)-f(sx)}{(c-1)x}=
%s\frac{f(c(sx))-f(sx)}{(c-1)sx}=s\Scal_s\J_cf(x).
%\end{align*}
The covariance relation is the multiplicative analogue of the usual derivative identity $\frac{d}{dx}f(sx)=sf'(sx)$.
It is not shared by an additive quotient with a fixed external length scale. Indeed, for
$$
\Delta_\delta f(x)=\frac{f(x+\delta)-f(x)}{\delta},
\qquad
\delta\neq0,
$$
one has
\begin{equation*}
\Delta_\delta\Scal_s=
s\Scal_s\Delta_{s\delta}.
\end{equation*}
Thus the additive increment must be rescaled whenever the coordinate is dilated.

Matching the quotient ratio to the intrinsic dilation gives more than covariance: it gives an exact finite representation.
\begin{proposition}[Exact closure at matched dilation scales]
\label{prop:matched-closure}
For every integer $m\geq1$ and every $x>0$,
\begin{equation}
\J_{b^m}W(x)=\frac{\left(a^{-m}-1\right)W(x)-a^{-m}\displaystyle\sum_{j=0}^{m-1}a^j\phi\left(b^j x\right)}{\left(b^m-1\right)x}.
\label{eq:matched-closure}
\end{equation}
In particular,
\begin{equation}
\J_bW(x)=\frac{\left(a^{-1}-1\right)W(x)-a^{-1}\phi(x)}{(b-1)x}.
\label{eq:matched-closed-force}
\end{equation}
\end{proposition}

\begin{proof}
Insert equation~\eqref{eq:scale-recursion-m} into the definition of $\J_{b^m}$.
\end{proof}

For every fixed $c>1$, the sequence $\J_cW_N$ converges locally uniformly because $W_N$ converges uniformly, so convergence alone does not single out $c=b$. The matched choices $c=b^m$ are distinguished at the canonical rough-component level by the finite closure in equation~\eqref{eq:matched-closure}. In the inverse problems below, the same intrinsic dilation is used to define a matched finite-scale quotient of the likelihood; no exact closure is asserted for the nonlinear likelihood itself.

The multiplicative nature is also transparent in logarithmic coordinates. With
$$
\Acal_cf(x)=x\J_cf(x)=\frac{f(cx)-f(x)}{c-1},
\qquad
g(y)=f(e^y),
$$
one has
$$
\Acal_cf(e^y)=\frac{g(y+\log c)-g(y)}{c-1}.
$$
Thus a fixed relative displacement in $x$ is a fixed translation in $\log x$.  The logarithmic representation clarifies the geometry, while the
Jackson quotient retains the same relative-scale structure directly in the
original variables.  This can be useful for black-box and coupled models,
where an explicit logarithmic reparameterization of the complete forward
map may be undesirable; see \cite{YangDeng2025}.

\subsection{Classical and finite-scale force refinement}
\label{subsec:opposite-refinement-laws}

This subsection contains the basic separation mechanism. Exact differentiation amplifies the newly resolved oscillations, whereas a fixed finite-scale quotient is continuous with respect to uniform potential approximation. We first prove the precise classical growth law and the multiplicative-quotient convergence estimate. The additive analogue then follows directly from the same uniform tail bound.

Assume first that
$$
\phi\in H^1_{\mathrm{per}}(0,1),
\qquad
\|\phi'\|_{L^2(0,1)}>0,
$$
and that
$$
0<a<1<ab,
\qquad
b>1.
$$
Since $W_N$ is a finite sum, its weak derivative is
$$
W_N'(x)
=
\sum_{n=0}^{N}
(ab)^n\phi'\left(b^nx\right).
$$

\begin{theorem}[Growth of classical-force increments]
\label{thm:classical-force-growth}
Under the assumptions above, define
\begin{equation}
\zeta_N
=
\left[
\frac{1}{b^N}
\int_0^{b^N}
|\phi'(u)|^2\,du
\right]^{1/2}.
\label{eq:classical-rhoN}
\end{equation}
Then
\begin{equation}
\left\|
W_N'-W_{N-1}'
\right\|_{L^2(0,1)}
=
(ab)^N\zeta_N,
\label{eq:classical-increment-exact}
\end{equation}
and
$$
\zeta_N\longrightarrow\|\phi'\|_{L^2(0,1)}.
$$
Consequently,
$$
\left\|
W_N'-W_{N-1}'
\right\|_{L^2(0,1)}
\sim
(ab)^N\|\phi'\|_{L^2(0,1)},
$$
so the adjacent increments diverge and the sequence
$\{W_N'\}_{N\geq0}$ is not Cauchy in $L^2(0,1)$. If
$b\in\N$ and $b\geq2$, then
$\zeta_N=\|\phi'\|_{L^2(0,1)}$ for every $N$.
\end{theorem}

\begin{proof}
The adjacent derivative increment is
$$
W_N'(x)-W_{N-1}'(x)
=
(ab)^N\phi'\left(b^Nx\right).
$$
The change of variables $u=b^Nx$ gives
$$
\left\|
W_N'-W_{N-1}'
\right\|_{L^2(0,1)}^2
=
(ab)^{2N}
\frac{1}{b^N}
\int_0^{b^N}|\phi'(u)|^2\,du,
$$
which proves equation~\eqref{eq:classical-increment-exact}. Since
$|\phi'|^2$ is one-periodic and integrable, its averages over $[0,T]$
converge as $T\to\infty$ to its mean over one period. Hence
$\zeta_N\to\|\phi'\|_{L^2(0,1)}$. The remaining conclusions follow from
$ab>1$ and $\|\phi'\|_{L^2(0,1)}>0$. For integer $b$, the interval
$[0,b^N]$ consists of exactly $b^N$ periods.
\end{proof}

We now return to the general bounded periodic profile. Neither differentiability of $\phi$ nor integrality of $b$ is needed for the multiplicative-difference estimates.

\begin{theorem}[Weighted and local convergence of Jackson forces]
\label{thm:jackson-force-convergence}
Let $\phi$ be bounded, let $0<a<1$, and fix $b>1$ and $c>1$. Then
\begin{equation}
\sup_{x>0}
x
\left|
\J_cW_N(x)-\J_cW_{N-1}(x)
\right|
\leq
\frac{2\|\phi\|_{L^\infty}}{c-1}
a^N.
\label{eq:weighted-jackson-increment}
\end{equation}
Moreover,
\begin{equation}
\sup_{x>0}
x
\left|
\J_cW_N(x)-\J_cW(x)
\right|
\leq
\frac{2\|\phi\|_{L^\infty}}{c-1}
\frac{a^{N+1}}{1-a}.
\label{eq:weighted-jackson-tail}
\end{equation}
Consequently, for every compact interval $I=[r,R]\subset(0,\infty)$,
\begin{equation}
\left\|
\J_cW_N-\J_cW
\right\|_{L^\infty(I)}
\leq
\frac{2\|\phi\|_{L^\infty}}{(c-1)r}
\frac{a^{N+1}}{1-a}.
\label{eq:local-jackson-tail}
\end{equation}
In particular, $\J_cW_N\to\J_cW$ locally uniformly on $(0,\infty)$.
\end{theorem}

\begin{proof}
Since
$$
W_N(x)-W_{N-1}(x)
=
a^N\phi\left(b^Nx\right),
$$
linearity gives
$$
\J_cW_N(x)-\J_cW_{N-1}(x)
=
\frac{a^N}{(c-1)x}
\left[
\phi\left(cb^Nx\right)-\phi\left(b^Nx\right)
\right].
$$
The boundedness of $\phi$ immediately yields equation~\eqref{eq:weighted-jackson-increment}. For the truncation error,
$$
\J_cW_N(x)-\J_cW(x)
=
\frac{
[W_N-W](cx)-[W_N-W](x)
}{(c-1)x},
$$
so the uniform potential tail equation~\eqref{eq:potential-tail} gives equation~\eqref{eq:weighted-jackson-tail}. Dividing by $x\geq r$ on $I$ proves equation~\eqref{eq:local-jackson-tail}.
\end{proof}

The preceding estimates isolate the refinement mismatch. Under the classical assumptions,
$$
\left\|W_N'-W_{N-1}'\right\|_{L^2(0,1)}
=
(ab)^N\zeta_N
\asymp
(ab)^N,
$$
whereas, for fixed $c>1$ and $I=[r,R]\subset(0,\infty)$,
$$
\left\|\J_cW_N-\J_cW_{N-1}\right\|_{L^\infty(I)}
\leq
\frac{2\|\phi\|_{L^\infty}}{(c-1)r}a^N.
$$
For the target family in equation~\eqref{eq:rough-potentials}, the baseline $V$ is independent of $N$, so the same rates hold for adjacent exact and split multiplicative proposal fields after multiplication by $|\kappa|$.

Convergence at a fixed finite scale is not unique to multiplicative differences. For
$$
\Delta_\delta f(x)=\frac{f(x+\delta)-f(x)}{\delta},
\qquad \delta\neq0,
$$
the uniform tail bound gives
$$
\left\|\Delta_\delta W_N-\Delta_\delta W\right\|_{L^\infty(\R)}
\leq
\frac{2\|\phi\|_{L^\infty}}{|\delta|}
\frac{a^{N+1}}{1-a}.
$$
The matched Jackson quotient is therefore distinguished structurally rather than exclusively: \eqref{eq:scale-covariance} gives compatibility with multiplicative rescaling, and Proposition~\ref{prop:matched-closure} gives exact finite closure when $c=b^m$.

\subsection{Hölder regularity and the nonclassical finite-scale regime}
\label{subsec:holder-finite-scale}

The preceding convergence estimates keep the multiplicative ratio $c$ fixed. This is essential in the rough regime: if a limiting potential is only Hölder continuous, sending $c$ to one need not produce a finite classical force. Thus $c$ is treated as a finite structural scale, preferably matched to the intrinsic dilation $b$, rather than as a discretization parameter that must vanish. At the same time, Hölder continuity supplies the modulus needed to control finite compositions of proposal maps.

To make the first point precise, let $F$ be locally Hölder continuous with exponent $\alpha\in(0,1]$. On a compact set $K\subset(0,\infty)$, suppose
$$
|F(x)-F(y)|
\leq
C_K|x-y|^\alpha.
$$
Whenever $x$ and $cx$ belong to $K$,
$$
|\J_cF(x)|
\leq
C_K(c-1)^{\alpha-1}x^{\alpha-1}.
$$
For $\alpha<1$, the factor $(c-1)^{\alpha-1}$ diverges as $c\to1$. This does not contradict the fixed-scale construction: a Hölder function need not possess a classical derivative, and the present method does not rely on recovering one through the limit $c\to1$.

For a Lipschitz periodic profile $\phi$, the Weierstrass family has a uniform Hölder modulus for every exponent $\alpha\in(0,1]$ satisfying
$$
ab^\alpha<1.
$$
Equivalently, one may take
$$
0<\alpha<\min\left\{1,-\frac{\log a}{\log b}\right\},
$$
with the endpoint $\alpha=1$ also allowed when $ab<1$. The precise estimate is derived in Appendix~\ref{app:holder}. On compact intervals away from zero, the matched closure formula then transfers this modulus to $\J_bW$. Combined with the regularity of the smooth baseline force, this provides the continuity input used later for finite proposal compositions.

\subsection{Posterior convergence under uniform likelihood approximation}
\label{subsec:target-convergence}

The force estimates above are model specific, whereas the probability-level implication is general. The next result identifies the topology in which the targets converge; its role is to contrast posterior stability with the stronger sensitivity convergence required by gradient-based sampling.

\begin{theorem}[Posterior stability]
\label{thm:posterior-stability}
Let $(\Theta,\mathcal B,\mu_0)$ be a probability space, and let
$$
\Lcal_N,\Lcal:\Theta\longrightarrow\R
$$
be measurable. Suppose
$$
\|\Lcal_N-\Lcal\|_{L^\infty(\mu_0)}
\leq
\varepsilon_N
$$
and
$$
0<Z
=
\int_\Theta
\exp[-\Lcal(\theta)]\,\mu_0(d\theta)
<\infty.
$$
Define
$$
Z_N
=
\int_\Theta
\exp[-\Lcal_N(\theta)]\,\mu_0(d\theta)
$$
and
$$
\frac{d\pi_N}{d\mu_0}
=
Z_N^{-1}\exp[-\Lcal_N],
\qquad
\frac{d\pi}{d\mu_0}
=
Z^{-1}\exp[-\Lcal].
$$
Then $0<Z_N<\infty$ and
$$
\exp(-\varepsilon_N)
\leq
\frac{Z_N}{Z}
\leq
\exp(\varepsilon_N).
$$
Moreover,
\begin{align}
\|\pi_N-\pi\|_{\TV}
&\leq
\frac12
\left[
\exp(2\varepsilon_N)-1
\right],
\label{eq:general-tv-bound}
\\
 d_{\Hell}(\pi_N,\pi)
&\leq
\frac{1}{\sqrt{2}}
\left[
\exp(\varepsilon_N)-1
\right],
\label{eq:general-hellinger-bound}
\end{align}
where
$$
d_{\Hell}^2(\pi_N,\pi)
=
\frac12
\int_\Theta
\left[
\sqrt{\frac{d\pi_N}{d\mu_0}}
-
\sqrt{\frac{d\pi}{d\mu_0}}
\right]^2
\mu_0(d\theta).
$$
For every bounded measurable quantity of interest $Q$, one also has
\begin{equation}
\left|
\E_{\pi_N}[Q]-\E_\pi[Q]
\right|
\leq
\|Q\|_{L^\infty}
\left[
\exp(2\varepsilon_N)-1
\right].
\label{eq:qoi-convergence}
\end{equation}
\end{theorem}

\begin{proof}
The uniform approximation gives
$$
-\varepsilon_N\leq\Lcal_N(\theta)-\Lcal(\theta)
\leq
\varepsilon_N
$$
for $\mu_0$-almost every $\theta$. Hence
$$
\exp(-\varepsilon_N)\exp[-\Lcal(\theta)]
\leq
\exp[-\Lcal_N(\theta)]
\leq
\exp(\varepsilon_N)\exp[-\Lcal(\theta)].
$$
Integrating with respect to $\mu_0$ yields $\exp(-\varepsilon_N)Z\leq Z_N\leq\exp(\varepsilon_N)Z$.
Therefore, $\exp(-2\varepsilon_N)\leq\frac{d\pi_N}{d\pi}\leq\exp(2\varepsilon_N)$.
It follows that
$$
\left|
\frac{d\pi_N}{d\pi}-1
\right|
\leq
\exp(2\varepsilon_N)-1.
$$
Integration against $\pi$ proves equation~\eqref{eq:general-tv-bound}. Also,
$$
\exp(-\varepsilon_N)
\leq
\sqrt{\frac{d\pi_N}{d\pi}}
\leq
\exp(\varepsilon_N),
$$
so
$$
\left|
\sqrt{\frac{d\pi_N}{d\pi}}-1
\right|
\leq
\exp(\varepsilon_N)-1.
$$
The Hellinger estimate follows by integration against $\pi$. Finally,
$$
\left|
\E_{\pi_N}[Q]-\E_\pi[Q]
\right|
\leq
\|Q\|_{L^\infty}
\|\pi_N-\pi\|_{L^1},
$$
which gives equation~\eqref{eq:qoi-convergence}.
\end{proof}

\begin{proposition}[Gaussian forward-map criterion]
\label{prop:gaussian-forward-criterion}
Let
$$
\mathcal G_N,\mathcal G:\Theta\longrightarrow\R^m
$$
be measurable forward maps, let $y\in\R^m$, and let $\Gamma$ be symmetric positive definite. Define
$$
\Lcal_N(\theta)
=
\frac12
\left|
\Gamma^{-1/2}
\left[
\mathcal G_N(\theta)-y
\right]
\right|^2,
$$
with $\Lcal$ defined analogously using $\mathcal G$. Suppose
$$
\|\mathcal G_N-\mathcal G\|_{L^\infty(\mu_0)}
\leq
\delta_N
$$
and
$$
\left\|
\max\left\{
|\mathcal G_N(\cdot)-y|,
|\mathcal G(\cdot)-y|
\right\}
\right\|_{L^\infty(\mu_0)}
\leq
M.
$$
Then
$$
\|\Lcal_N-\Lcal\|_{L^\infty(\mu_0)}
\leq
M\|\Gamma^{-1}\|_2\delta_N.
$$
Consequently, all conclusions of Theorem~\ref{thm:posterior-stability} hold with
$$
\varepsilon_N
=
M\|\Gamma^{-1}\|_2\delta_N.
$$
\end{proposition}

\begin{proof}
Set
$$
r_N(\theta)=\mathcal G_N(\theta)-y,
\qquad
r(\theta)=\mathcal G(\theta)-y.
$$
Then
\begin{align*}
|\Lcal_N(\theta)-\Lcal(\theta)|
&=
\frac12
\left|
r_N(\theta)^\top\Gamma^{-1}r_N(\theta)
-
r(\theta)^\top\Gamma^{-1}r(\theta)
\right|
\\
&\leq
\frac12
\|\Gamma^{-1}\|_2
|r_N(\theta)-r(\theta)|
\left[
|r_N(\theta)|+|r(\theta)|
\right]
\\
&\leq
M\|\Gamma^{-1}\|_2\delta_N.
\end{align*}
The posterior bounds follow from Theorem~\ref{thm:posterior-stability}.
\end{proof}

For the rough family in Section~\ref{subsec:rough-target-family}, take
$$
\mu_V(dx)
=
Z_V^{-1}\exp[-V(x)]\,dx
$$
as the reference probability measure and set
$$
\Lcal_N(x)=\kappa W_N(x),
\qquad
\Lcal(x)=\kappa W(x).
$$
By equation~\eqref{eq:potential-tail},
$$
\|\Lcal_N-\Lcal\|_{L^\infty(\mu_V)}
\leq
\eta_N,
\qquad
\eta_N
=
|\kappa|
\frac{\|\phi\|_{L^\infty}}{1-a}
a^{N+1}.
$$
Therefore,
$$
\|\pi_N-\pi\|_{\TV}
\leq
\frac12
\left[
\exp(2\eta_N)-1
\right]
=
O(a^N),
$$
and the Hellinger distance and bounded quantity-of-interest errors have the same geometric order. No differentiability is required. Combined with Theorem~\ref{thm:classical-force-growth}, this yields a sequence of posterior targets converging in standard probability metrics even though their exact classical forces are not Cauchy in $L^2(0,1)$.

The same theorem also closes the probability-level part of Proposition~\ref{prop:wiggly-separation}. Let $\mu_V(dx)=Z_V^{-1}e^{-V(x)}dx$ and $\Lcal_n=R_n$. Then
$$
\|\Lcal_n\|_{L^\infty(\mu_V)}
\leq
r_n\|\psi\|_{L^\infty},
$$
so, writing $\pi_n$ for the target with potential $V+R_n$ and $\pi_V$ for the target with potential $V$,
\begin{align*}
\|\pi_n-\pi_V\|_{\TV}
&\leq
\frac12
\left[
\exp\left(2r_n\|\psi\|_{L^\infty}\right)-1
\right],
\\
d_{\Hell}(\pi_n,\pi_V)
&\leq
\frac{1}{\sqrt2}
\left[
\exp\left(r_n\|\psi\|_{L^\infty}\right)-1
\right].
\end{align*}
Hence $r_n=\varepsilon_n^\alpha$, $0<\alpha<1$, yields an explicit sequence for which the posterior error is $O(\varepsilon_n^\alpha)$ while the exact-force discrepancy is of order $\varepsilon_n^{\alpha-1}$.

\section{Resolution-stable Metropolized finite-scale dynamics}
\label{sec:resolution-kernels}

This section lifts field convergence to deterministic proposals and Metropolis kernels. Exactness follows from the shear structure and the use of the true finite-resolution Hamiltonian in the acceptance step; it is independent of how the proposal field is constructed. Local field convergence is then propagated through a fixed number of leapfrog stages and contrasted with the first classical half-kick.

\subsection{Measurable shear proposals and exactness}

We first isolate the measure-theoretic mechanism behind Metropolis correction. The key point is that differentiability of the proposal force is unnecessary: kick and drift maps are triangular translations, so their volume preservation follows directly from translation invariance.

For this abstract exactness argument, let
$$
U_N:\R^d\longrightarrow(-\infty,\infty]
$$
denote a measurable finite-resolution potential. This allows hard support
constraints to be represented by $U_N=+\infty$ outside the admissible
parameter domain. When $d=1$, the rough family in
equation~\eqref{eq:rough-potentials} is the principal example. Let the phase
space be
$$
\Zcal=\R^d\times\R^d,
\qquad
z=(x,p),
$$
and let $K:\R^d\to\R$ be a measurable kinetic energy. Define
$$
\Hcal_N(x,p)
=
U_N(x)+K(p),
\qquad
w_N(z)=\exp[-\Hcal_N(z)],
$$
with the convention $\exp(-\infty)=0$, and assume
$$
0<
\widetilde Z_N
=
\int_{\R^{2d}}w_N(z)\,dz
<\infty.
$$
The extended target is
$$
\widetilde\pi_N(dz)
=
\frac{w_N(z)}{\widetilde Z_N}\,dz.
$$

Let $G_N:\R^d\to\R^d$ be a measurable proposal field and $v:\R^d\to\R^d$ a measurable momentum velocity. The construction applies to exact-gradient, additive finite-scale, and multiplicative finite-scale fields. Define
$$
\Pcal_{N,h}(x,p)
=
\left(x,p-hG_N(x)\right)
$$
and
$$
\Qcal_h(x,p)
=
\left(x+hv(p),p\right).
$$
The one-step leapfrog map and the momentum reversal are
$$
\Phi_{N,h}
=
\Pcal_{N,h/2}
\circ
\Qcal_h
\circ
\Pcal_{N,h/2},
\qquad
\Rcal(x,p)
=
(x,-p),
$$
and the $L$-step deterministic proposal is
$$
\Psi_N
=
\Rcal\circ\Phi_{N,h}^{L}.
$$

The kick and drift are measurable triangular translations. If $G_N$ and $v$ are finite almost everywhere, their inverses are
$$
\Pcal_{N,h}^{-1}
=
\Pcal_{N,-h},
\qquad
\Qcal_h^{-1}
=
\Qcal_{-h}.
$$
They also preserve Lebesgue measure. For example, Tonelli's theorem and translation invariance give, for every nonnegative measurable $f$,
\begin{align*}
\int_{\R^{2d}}
f\left(\Pcal_{N,h}(x,p)\right)\,dp\,dx
&=
\int_{\R^d}\int_{\R^d}
f\left(x,p-hG_N(x)\right)\,dp\,dx
\\
&=
\int_{\R^d}\int_{\R^d}
f(x,\widetilde p)\,d\widetilde p\,dx.
\end{align*}
The drift is handled in the same way. This elementary shear structure is the only volume-preservation input needed below.

Assume

\begin{equation}
K(-p)
=
K(p),
\qquad
v(-p)
=
-v(p).
\label{eq:parity}
\end{equation}
Then
$$
\Rcal\Phi_{N,h}\Rcal
=
\Phi_{N,h}^{-1},
$$
and $\Psi_N$ is a measure-preserving involution.
For $w_N(z)>0$, define
\begin{equation}
\alpha_N(z)
=
1\wedge
\frac{w_N(\Psi_Nz)}{w_N(z)}.
\label{eq:metropolis-weight}
\end{equation}
On the zero-density set $\{w_N=0\}$, the value of $\alpha_N$ may be chosen
arbitrarily; we set it equal to one. For finite Hamiltonian values,
equation~\eqref{eq:metropolis-weight} is the usual expression
$$
1\wedge
\exp\left[
-\Hcal_N(\Psi_Nz)+\Hcal_N(z)
\right].
$$

\begin{theorem}[Exactness with measurable proposal fields]
\label{thm:rough-exactness}
Suppose $0<\widetilde Z_N<\infty$, $G_N$ and $v$ are measurable and finite almost everywhere, and equation~\eqref{eq:parity} holds. Then the deterministic Metropolis kernel
\begin{equation*}
M_N(z,dz')=\alpha_N(z)\delta_{\Psi_Nz}(dz')+
\left[1-\alpha_N(z)\right]\delta_z(dz')
\end{equation*}
is reversible with respect to $\widetilde\pi_N$. No differentiability of $U_N$, $G_N$, $K$, or $v$ is required.
\end{theorem}

\begin{proof}
By the preceding shear calculation, $\Psi_N$ is measure preserving, and the
parity condition makes it an involution. By
equation~\eqref{eq:metropolis-weight},
$$
w_N(z)\alpha_N(z)
=
\min\left\{
w_N(z),w_N(\Psi_Nz)
\right\},
$$
including points at which one of the two weights vanishes. This expression is
invariant under $z\leftrightarrow\Psi_Nz$. The accepted flow is symmetric,
and the rejected part is diagonal.
\end{proof}
Momentum refreshment followed by $M_N$ gives the corresponding position-space Metropolized Hamiltonian kernel. Since refreshment preserves the kinetic factor, the resulting chain has position marginal proportional to $\exp[-U_N]$; for the one-dimensional family in Section~\ref{subsec:rough-target-family}, this marginal is exactly $\pi_N$.

\subsection{Finite-scale split proposal fields}

Having established exactness for an arbitrary measurable field, we return to the one-dimensional target family of Section~\ref{subsec:rough-target-family}. When the baseline potential $V$ is differentiable, the numerical comparisons use the split fields
$$
G_N^{\mathrm{add}}(x)
=
V'(x)+\kappa\Delta_\delta W_N(x),
$$
$$
G_N^{(c)}(x)
=
V'(x)+\kappa\J_cW_N(x),
\qquad c>1,
$$
where $\delta\neq0$ is fixed. The choice $c=b$ gives the matched Jackson field, while $c\neq b^m$ gives a multiplicative but nonresonant field. A full multiplicative field $\J_cU_N$ could also be used, but the split form retains the exact derivative of the smooth baseline and regularizes only the rough component.

All these choices produce exact Metropolis kernels by Theorem~\ref{thm:rough-exactness}. Exactness depends on the triangular shear structure and on evaluating the true finite-resolution Hamiltonian in the acceptance probability; it does not require the proposal field to equal $U_N'$. Moreover, for every fixed $\delta\neq0$ and $c>1$, the uniform convergence of $W_N$ implies local uniform convergence of the corresponding finite-scale fields. We write a generic convergent field as
$$
G_N\longrightarrow G.
$$
When $G$ is continuous, the limiting kick, leapfrog map, and involutive proposal are
$$
\Pcal_h(x,p)
=
\left(x,p-hG(x)\right),
$$
$$
\Phi_h
=
\Pcal_{h/2}\circ\Qcal_h\circ\Pcal_{h/2},
$$
and
$$
\Psi
=
\Rcal\circ\Phi_h^L.
$$
The convergence results below therefore apply to any fixed additive or multiplicative quotient. Scale matching enters through the covariance and closure properties proved in equation~\eqref{eq:scale-covariance} and Proposition~\ref{prop:matched-closure}, not through exactness or convergence alone.
The motivating family in Section~\ref{subsec:wiggly-generalization} shows why a nonzero-scale proposal field is needed, but it carries no distinguished multiplicative ratio. The algorithmic emphasis below is therefore on the nested Weierstrass setting, where the model itself supplies the dilation scale $b$ and the matched Jackson field inherits the covariance and closure properties established in equation~\eqref{eq:scale-covariance} and Proposition~\ref{prop:matched-closure}.

\subsection{Convergence of proposal maps and acceptance probabilities}

Local field convergence yields a finite-trajectory limit provided that all intermediate stages remain in one common compact set.
\begin{definition}[Admissible compact set]
\label{def:admissible-compact}
Fix $h>0$ and $L\in\N$. A compact set $C\subset\Zcal$
is called admissible for the sequence $\left\{\Phi_{N, h}\right\}_{N\geq0}$
and the limiting map $\Phi_h$ if there exist a compact set $C_\ast\subset\Zcal$
and an integer $N_0$ such that all full-step and intermediate-stage states generated from $C$ during the first $L$ leapfrog steps remain in $C_\ast$, uniformly for $N\geq N_0$ and for the limiting dynamics. More precisely, for every $N\geq N_0$
and every $\ell=0, \ldots, L-1$,
we require
\begin{align*}
&\Phi_{N,h}^{\ell}(C)\subset C_\ast,\\
&\Pcal_{N,h/2}\left(\Phi_{N,h}^{\ell}(C)\right)\subset C_\ast,\\
&\Qcal_h\circ\Pcal_{N,h/2}\left(\Phi_{N,h}^{\ell}(C)\right)\subset C_\ast,
\end{align*}
together with the corresponding conditions for the limiting maps,
\begin{align*}
&\Phi_h^{\ell}(C)\subset C_\ast,
\\
&\Pcal_{h/2}\left(
\Phi_h^{\ell}(C)
\right)\subset C_\ast,
\\
&\Qcal_h\circ
\Pcal_{h/2}
\left(
\Phi_h^{\ell}(C)
\right)
\subset C_\ast.
\end{align*}
We also require the terminal full-step states to satisfy
\begin{equation*}
\Phi_{N,h}^{L}(C)
\cup
\Phi_h^{L}(C)
\subset
C_\ast,
\qquad
N\geq N_0.
\end{equation*}
Finally, we require
$$
\Rcal(C_\ast)=C_\ast.
$$
\end{definition}
This is a finite-time compact-containment condition. It guarantees that the force convergence and the uniform-continuity estimates used below are applied on one common compact set containing every intermediate stage of the finite-resolution and limiting leapfrog trajectories. If the position variable is restricted to the positive orthant, admissibility additionally requires that the position projection of $C_\ast$ be compactly contained in that orthant.

In particular, admissibility excludes finite-time escape. On the positive half-line it also keeps trajectories away from the coordinate singularity at $x=0$; a log-coordinate implementation provides an alternative way to enforce positivity.

\begin{theorem}[Finite-length proposal convergence]
\label{thm:proposal-convergence}
Let $G_N\to G$ locally uniformly, and assume $G$ and $v$ are continuous. Fix $h>0$, $L\in\N$, and an admissible compact set $C$. Then
\begin{equation}
\sup_{z\in C}
\left|
\Phi_{N,h}^{L}(z)-\Phi_h^{L}(z)
\right|
\longrightarrow
0,
\label{eq:proposal-convergence}
\end{equation}
and hence
\begin{equation}
\sup_{z\in C}
\left|
\Psi_N(z)-\Psi(z)
\right|
\longrightarrow
0.
\label{eq:involution-convergence}
\end{equation}
\end{theorem}

\begin{proof}
On the position projection of $C_\ast$, local uniform convergence of $G_N$ gives uniform convergence of the finite-resolution kicks to the limiting kick. The drift is independent of $N$ and is uniformly continuous on the relevant momentum projection. The admissibility condition ensures that these estimates can be applied at every intermediate stage of the first $L$ steps. It follows that one full leapfrog map converges uniformly on each stage set. Induction over the fixed number of steps, using uniform continuity of the limiting kick and drift on $C_\ast$, proves equation~\eqref{eq:proposal-convergence}. The momentum flip is an isometry, which gives equation~\eqref{eq:involution-convergence}.
\end{proof}

The preceding theorem is qualitative and requires no Lipschitz regularity of the limiting field. We next combine proposal convergence with local uniform convergence of the potentials to control the Hamiltonian defect and the Metropolis kernel.

\begin{theorem}[Acceptance and Metropolis-kernel convergence]
\label{thm:kernel-convergence}
Assume the hypotheses of Theorem~\ref{thm:proposal-convergence}. Suppose also that
$$
U_N
\longrightarrow
U
$$
locally uniformly and that $U$ and $K$ are continuous. Define
$$
\Hcal(x,p)
=
U(x)+K(p),
$$
and let $\alpha$ and $M$ be the limiting acceptance function and deterministic Metropolis kernel obtained from $\Hcal$ and $\Psi$. Then on every admissible compact set $C$,

\begin{equation}
\sup_{z\in C}
\left|
\alpha_N(z)-\alpha(z)
\right|
\longrightarrow
0.
\label{eq:acceptance-convergence}
\end{equation}

For every bounded Lipschitz function $f$,
\begin{equation*}
\sup_{z\in C}
\left|
M_Nf(z)-Mf(z)
\right|
\longrightarrow
0.
\end{equation*}
Moreover, if
$$
d_{\rm BL}(\nu_1,\nu_2)
=
\sup_{\|f\|_\infty+\operatorname{Lip}(f)\leq1}
\left|
\int f\,d\nu_1-\int f\,d\nu_2
\right|,
$$
then
$$
\sup_{z\in C}
d_{\rm BL}\bigl(M_N(z,\cdot),M(z,\cdot)\bigr)
\longrightarrow0.
$$
\end{theorem}

\begin{proof}
By Theorem~\ref{thm:proposal-convergence}, $\Psi_N\to\Psi$ uniformly on $C$. The images lie in the common compact set $C_\ast$. Local uniform convergence of $U_N$, together with continuity of $U$ and $K$, gives

$$
\sup_{z\in C}
\left|
\Hcal_N(\Psi_Nz)-\Hcal(\Psi z)
\right|
\longrightarrow
0
$$

and

$$
\sup_{z\in C}
\left|
\Hcal_N(z)-\Hcal(z)
\right|
\longrightarrow
0.
$$

The map $r\mapsto1\wedge e^{-r}$ is globally Lipschitz, so equation~\eqref{eq:acceptance-convergence} follows. Finally,

\begin{align*}
M_Nf(z)-Mf(z)
&=
\alpha_N(z)
\left[
f(\Psi_Nz)-f(\Psi z)
\right]
\\
&\quad+
\left[
\alpha_N(z)-\alpha(z)
\right]
\left[
f(\Psi z)-f(z)
\right],
\end{align*}

and the bounded-Lipschitz convergence follows.
\end{proof}

Because the limiting acceptance function is continuous and strictly positive on an admissible compact set $C$, uniform convergence implies the existence of $N_C$ and $\alpha_C>0$ such that
$$
\inf_{N\geq N_C}\inf_{z\in C}\alpha_N(z)\geq\alpha_C.
$$
If, in addition, $\inf_{N\geq N_0}\widetilde\pi_N(C)\geq\beta>0$, then
$$
\liminf_{N\to\infty}\int\alpha_N(z)\,\widetilde\pi_N(dz)
\geq\beta\alpha_C>0.
$$
This is a local refinement-stability statement, not a global mixing result.

\paragraph{Scope of the kernel limit.}
The results concern a fixed number of proposal steps on admissible compact sets. They do not imply a resolution-uniform spectral gap or mixing time, which would require additional global contraction, drift--minorization, or spectral estimates; see \cite{RudolfSchweizer2018}.

\subsection{Classical kick degeneration and the necessary step-size scale}

We close the section by locating the classical obstruction at the earliest stage of the integrator: the first half-kick. Suppose $V$ is differentiable and define

$$
U_N'(x)
=
V'(x)+\kappa W_N'(x).
$$

For a step size $h>0$, the first classical half-kick at resolution $N$ is

\begin{equation*}
\Kcal_{N,h}^{\mathrm{cl}}(x,p)
=
\left(
 x,
 p-\frac{h}{2}U_N'(x)
\right).
\end{equation*}

Let $\gamma$ be any probability measure on momentum space. Equip $(0,1)\times\R$ with $dx\otimes\gamma(dp)$.

\begin{theorem}[Nonconvergence of the classical kick]
\label{thm:classical-kick-degeneration}
Assume $\phi\in H^1_{\mathrm{per}}(0,1)$, $\|\phi'\|_{L^2(0,1)}>0$, $b>1$, $0<a<1<ab$, and $\kappa\neq0$. Let $\zeta_N$ be defined by equation~\eqref{eq:classical-rhoN}. Then

\begin{equation*}
\left\|
\Kcal_{N,h_N}^{\mathrm{cl}}
-
\Kcal_{N-1,h_N}^{\mathrm{cl}}
\right\|_{L^2(dx\otimes\gamma)}
=
\frac{|\kappa|h_N}{2}
(ab)^N\zeta_N.
\end{equation*}

Consequently:

\begin{enumerate}[label=\textup{(\roman*)},leftmargin=2.6em]
\item for every fixed $h>0$, the classical kick maps are not Cauchy in $L^2$;
\item a necessary condition for adjacent-resolution consistency is

\begin{equation*}
h_N(ab)^N
\longrightarrow
0.
\end{equation*}
\end{enumerate}
\end{theorem}

\begin{proof}
The position components coincide. Both adjacent resolutions are compared at the common step size $h_N$. The momentum-component difference is

$$
-\frac{\kappa h_N}{2}
\left[
W_N'(x)-W_{N-1}'(x)
\right].
$$

The result follows from Theorem~\ref{thm:classical-force-growth}; integration over $p$ contributes the factor one.
\end{proof}

For any fixed multiplicative ratio $c>1$, the corresponding split kick
$$
\Kcal_N^{(c)}(x,p)
=
\left(
 x,
 p-\frac{h}{2}
 \left[
 V'(x)+\kappa\J_cW_N(x)
 \right]
\right)
$$
is geometrically consistent on compact intervals away from zero. Indeed, for $I=[r,R]\subset(0,\infty)$,
$$
\sup_{(x,p)\in I\times\R}
\left|
\Kcal_N^{(c)}(x,p)
-
\Kcal_{N-1}^{(c)}(x,p)
\right|
\leq
\frac{|\kappa|h\|\phi\|_{L^\infty}}{(c-1)r}
a^N.
$$
The fixed additive kick has the same $O(ha^N)$ refinement behavior. Thus the algorithmic contrast is
$$
\boxed{
\begin{aligned}
\text{classical half-kick increment}
&\asymp
h_N(ab)^N,
\\
\text{fixed finite-scale half-kick increment}
&=
O\left(ha^N\right).
\end{aligned}
}
$$
The matched choice $c=b$ is not needed for this convergence estimate; it is distinguished by the dilation covariance and exact closure established in equation~\eqref{eq:scale-covariance} and Proposition~\ref{prop:matched-closure}.

\section{Algorithms and numerical experiments}
\label{sec:algorithms}

The experiments focus on the scale-structured setting for which the Jackson construction is intrinsic. They address three questions: whether the posterior or forward model stabilizes under nested refinement, whether exact differentiation remains stable at the same levels, and how the finite-scale proposal fields behave under the same refinement. Exact-gradient, additive, and mismatched multiplicative fields are retained as comparison baselines, and every method uses the same finite-resolution target potential in the Metropolis test.

\subsection{Comparison protocol}
\label{subsec:comparison-protocol}

Write
$$
U_N(\theta)
=
V(\theta)+\Lcal_N(\theta),
$$
where $V$ is the smooth prior potential and $\Lcal_N$ is the finite-resolution negative log-likelihood. We compare
\begin{align*}
G_N^{\mathrm{grad}}(\theta)
&=
\nabla U_N(\theta),
\\
G_N^{\mathrm{add}}(\theta)
&=
\nabla V(\theta)+\Delta_\delta\Lcal_N(\theta),
\\
G_N^{\mathrm{mis}}(\theta)
&=
\nabla V(\theta)+\J_{\sqrt b}\Lcal_N(\theta),
\\
G_N^{\mathrm{match}}(\theta)
&=
\nabla V(\theta)+\J_b\Lcal_N(\theta).
\end{align*}
In several dimensions, each quotient is applied coordinatewise. For example,
$$
\left[\J_c\Lcal_N(\theta)\right]_j
=
\frac{
\Lcal_N\left(\theta+(c-1)\theta_je_j\right)-\Lcal_N(\theta)
}{(c-1)\theta_j}.
$$
The additive increment is fixed across refinement levels. The choice
$c=\sqrt b$ preserves multiplicative displacement but is not matched to the
discrete recursion, whereas $c=b$ uses the intrinsic dilation of the
Weierstrass hierarchy. Here ``matched'' refers to this dilation ratio; the
nonlinear likelihood is not assumed to satisfy the exact closure of
Proposition~\ref{prop:matched-closure}.

The target is supported on the positive parameter domain. To preserve the
involutive kick--drift--kick proposal at the boundary, the proposal fields are
evaluated through finite extensions of the prior gradient and forward model to
the ambient Euclidean space, while the target potential is set to $+\infty$
outside its support. The leapfrog trajectory is therefore completed without
early truncation, and an endpoint outside the support is rejected by the
Metropolis test.

All chains use standard Gaussian momenta and a symmetric kick--drift--kick proposal. The step size is selected from a fixed logarithmic candidate grid by a pilot chain whose acceptance is closest to $0.7$. The two one-dimensional inverse problems use 12 independent chains, 3500 retained samples, 800 burn-in iterations, and pilot runs of 300 retained samples after 120 burn-in iterations. The scale problem uses five leapfrog steps and the porous-medium problem uses three; their additive increment is $\delta=0.08$. The Darcy experiment uses eight independent chains, 1800 retained samples, 500 burn-in iterations, three leapfrog steps, and pilot runs of 180 retained samples after 80 burn-in iterations, with $\delta=0.06$.

The reported diagnostics are posterior or forward convergence, sensitivity behavior, pilot-selected step sizes, and repeated-chain acceptance.

\subsection{Bayesian inversion of a microstructure scale}
\label{subsec:scale-inverse-experiment}

All scale-structured experiments below use the periodic profile $\phi(t)=\cos(2\pi t)$. The first inverse problem observes a discretely scale-invariant signal,
$$
y_j
=
W_{N_{\mathrm{ref}}}(t_j\theta^\dagger)+\eta_j,
\qquad
\eta_j\sim\mathcal N(0,\sigma^2),
$$
with observation times $t_j\in\{0.55,0.80,1.05,1.30,1.55,1.80\}$, $N_{\mathrm{ref}}=80$, $\theta^\dagger=1$, $a=0.8$, $b=1.4$, and $\sigma=0.5$. A Gaussian prior with mean $1$ and standard deviation $0.025$ is restricted to the positive half-line. The finite-resolution potential is
$$
U_N(\theta)
=
\frac{(\theta-1)^2}{2(0.025)^2}
+
\frac{1}{2\sigma^2}
\sum_{j=1}^{6}
\left[W_N(t_j\theta)-y_j\right]^2.
$$
Direct quadrature supplies a reference posterior at each level.

The two panels in Figure~\ref{fig:scale-posterior-acceptance} have different roles. Figure~\ref{fig:scale-posterior} checks convergence of the target distribution, whereas Figure~\ref{fig:scale-acceptance} records acceptance after resolution-dependent pilot tuning.

\begin{figure}[H]
\centering
\begin{subfigure}[t]{0.49\textwidth}
\centering
\includegraphics[width=\textwidth]{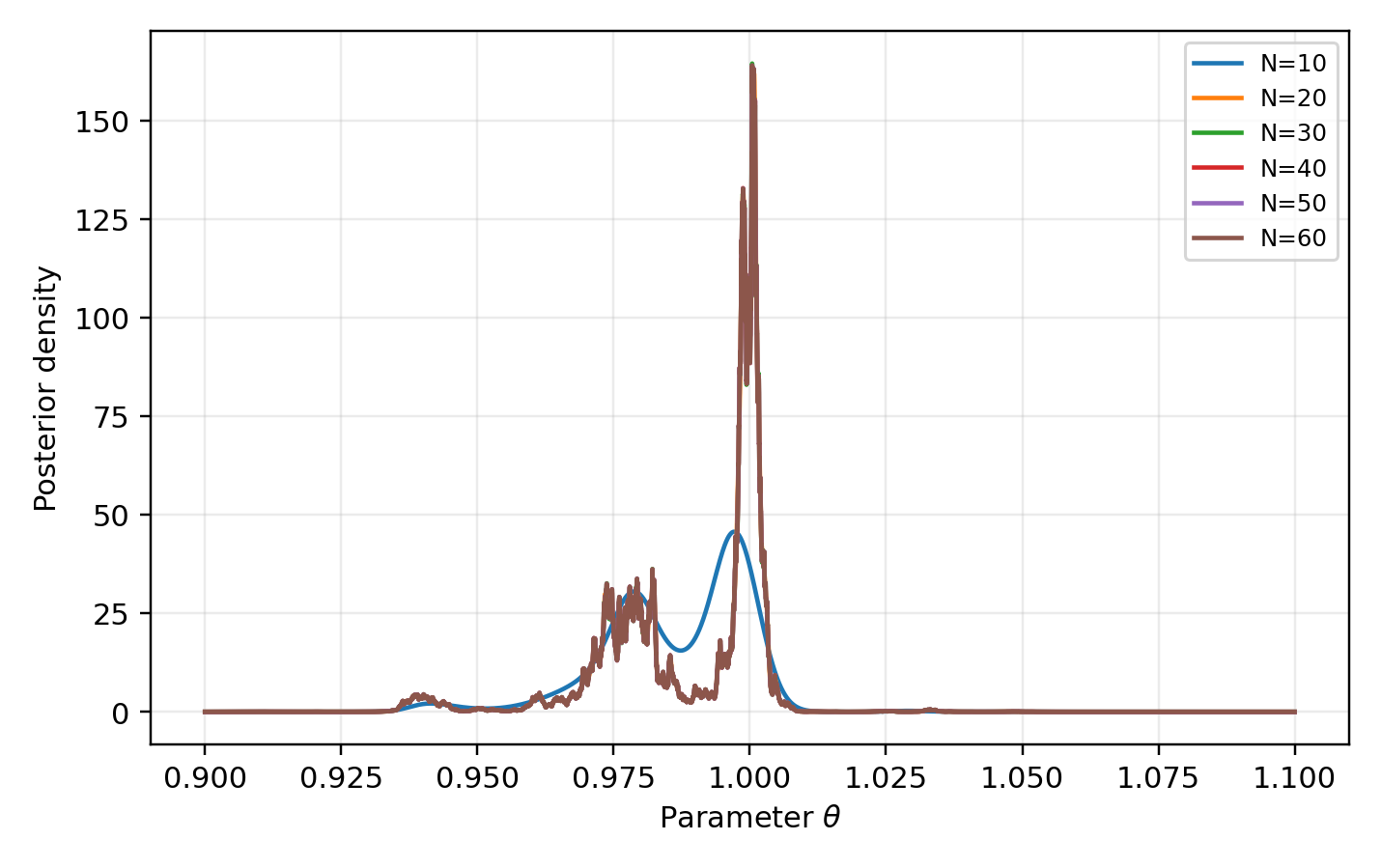}
\caption{Posterior densities.}
\label{fig:scale-posterior}
\end{subfigure}
\hfill
\begin{subfigure}[t]{0.49\textwidth}
\centering
\includegraphics[width=\textwidth]{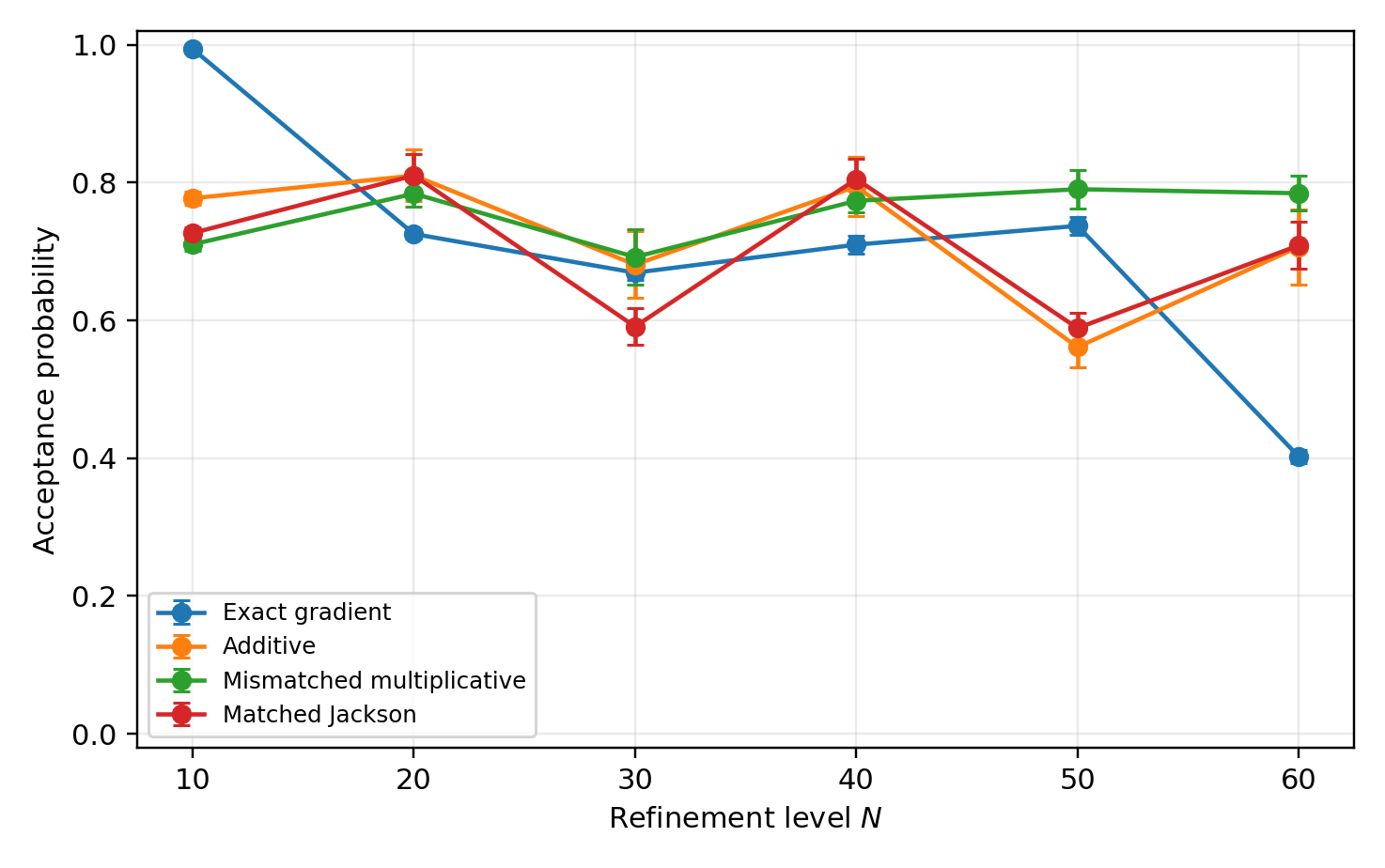}
\caption{Pilot-tuned acceptance.}
\label{fig:scale-acceptance}
\end{subfigure}
\caption{Bayesian microstructure-scale inversion. The posterior is effectively stable from $N=20$ onward. The right panel shows acceptance after pilot step-size selection. Error bars show one standard deviation over 12 independent chains.}
\label{fig:scale-posterior-acceptance}
\end{figure}

As seen in Figure~\ref{fig:scale-posterior}, the quadrature mean changes from $0.98659$ at $N=10$ to $0.98960$ at $N=20$ and then remains stable to the displayed precision; the posterior standard deviation similarly stabilizes near $0.01469$. Nevertheless, the pilot-selected exact-gradient step size decreases from $1.5\times10^{-3}$ at $N=10$ to $2.0\times10^{-5}$ at $N=60$. The target has therefore stabilized while the exact-gradient proposal still requires a factor-$75$ reduction in step size.

The finest-level proposal scales and acceptance rates are summarized in Table~\ref{tab:scale-tuning}. The three finite-scale fields retain acceptance between $0.707$ and $0.785$ with step sizes between $2.36\times10^{-4}$ and $4.37\times10^{-4}$, whereas exact-gradient HMC uses $2.00\times10^{-5}$ and still has acceptance $0.402$. Together with Figure~\ref{fig:scale-acceptance}, this supports a resolution-tuning distinction between exact differentiation and fixed finite-scale fields.

\begin{table}[H]
\centering
\small
\caption{Pilot-selected step sizes and repeated-chain acceptance at $N=60$ for the Bayesian scale inverse problem.}
\label{tab:scale-tuning}
\begin{tabular}{lrr}
\toprule
Method & Step size & Acceptance \\
\midrule
Exact gradient & $2.00\times10^{-5}$ & $0.402$ \\
Additive & $4.37\times10^{-4}$ & $0.707$ \\
Mismatched multiplicative & $2.36\times10^{-4}$ & $0.785$ \\
Matched Jackson & $4.37\times10^{-4}$ & $0.709$ \\
\bottomrule
\end{tabular}
\end{table}

\FloatBarrier
\subsection{One-dimensional porous-medium inverse problem}
\label{subsec:porous-experiment}

We next consider a one-dimensional elliptic flow model for a hierarchical porous medium.  Periodic porous structures with multiple separated physical scales are routinely treated by multiscale and reiterated homogenization \cite{HeWangPindera2019,NiuShenXu2020}.  Our coefficient is a stylized scale-geometric proxy for that setting: the PDE, rather than the posterior itself, contains the nested microstructure,
\begin{align*}
-\frac{d}{dx}
\left[k_N(x;\theta)\frac{dp_N}{dx}(x;\theta)\right]
&=1,
\qquad x\in(0,1),
\\
p_N(0;\theta)&=p_N(1;\theta)=0,
\end{align*}
with
$$
k_N(x;\theta)
=
\exp\left\{\rho W_N\left(\theta(0.3+x)\right)\right\}.
$$
The positive parameter $\theta$ dilates the common geometric hierarchy and can be interpreted as an uncertain calibration of the characteristic microstructure scale.  This is the nested analogue of the single periodic calibration parameter in Section~\ref{subsec:pde-origin}.
We use $a=0.8$, $b=1.4$, $\rho=0.35$, true parameter $\theta^\dagger=1$, and reference level $N_{\mathrm{ref}}=60$. Pressure is observed at $x_j\in\{0.15,0.29,0.43,0.57,0.71,0.85\}$ with independent Gaussian noise of standard deviation $0.002$. A Gaussian prior centered at $1$ with standard deviation $0.05$ is restricted to the positive half-line. The forward solution has an explicit flux representation, and the exact finite-resolution sensitivity is obtained by differentiating the same quadrature formula used in the forward solver.

As in the scale problem, Figure~\ref{fig:porous1d-posterior-acceptance} separates target convergence from resolution-dependent tuning. Figure~\ref{fig:porous1d-posterior} displays the posterior densities, and Figure~\ref{fig:porous1d-acceptance} displays the pilot-tuned acceptance rates.

\begin{figure}[H]
\centering
\begin{subfigure}[t]{0.49\textwidth}
\centering
\includegraphics[width=\textwidth]{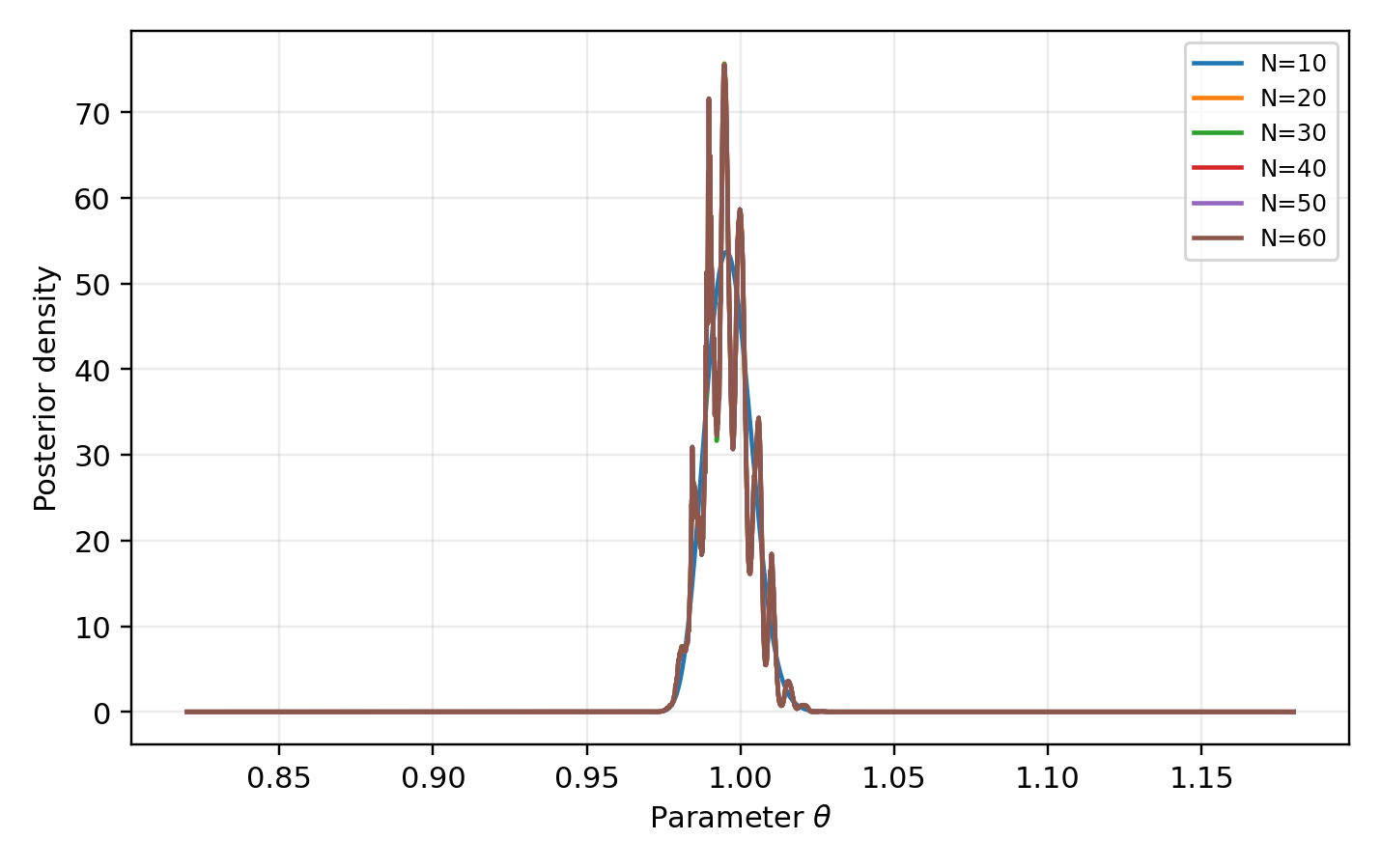}
\caption{Posterior densities.}
\label{fig:porous1d-posterior}
\end{subfigure}
\hfill
\begin{subfigure}[t]{0.49\textwidth}
\centering
\includegraphics[width=\textwidth]{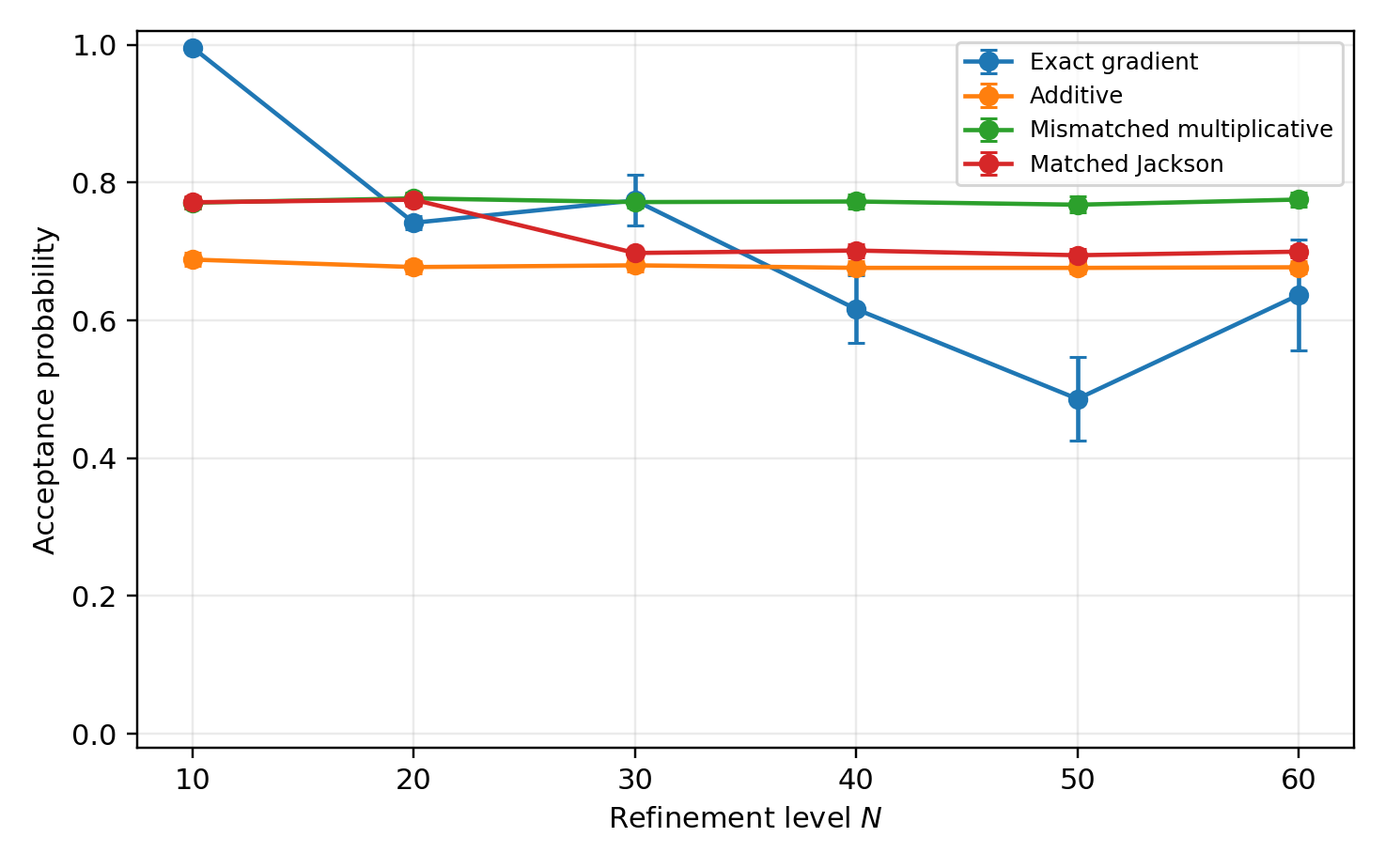}
\caption{Pilot-tuned acceptance.}
\label{fig:porous1d-acceptance}
\end{subfigure}
\caption{One-dimensional porous-medium inversion. The posterior stabilizes rapidly, while the exact-gradient proposal requires a strongly decreasing step size. Error bars show one standard deviation over 12 independent chains.}
\label{fig:porous1d-posterior-acceptance}
\end{figure}

Figure~\ref{fig:porous1d-posterior} shows that the quadrature mean changes from $0.99634$ at $N=10$ to $0.99595$ at $N=20$ and remains near $0.99594$ thereafter; the posterior standard deviation stabilizes near $0.00744$. In contrast, the exact-gradient step size decreases from $2.0\times10^{-3}$ at $N=10$ to $3.0\times10^{-5}$ at $N=60$. Pilot tuning keeps the finest-level exact-gradient acceptance at $0.637$, but only after reducing the step size by a factor of about $67$.

Table~\ref{tab:porous1d-tuning} shows that the finite-scale fields use substantially larger finest-level steps while maintaining comparable acceptance. The matched field uses the largest tested step, $2.00\times10^{-3}$, with acceptance $0.700$; the additive and mismatched fields use $6.02\times10^{-4}$ and $1.81\times10^{-4}$, respectively. This result is consistent with the intrinsic dilation being well aligned with this particular likelihood.

\begin{table}[H]
\centering
\small
\caption{Pilot-selected step sizes and repeated-chain acceptance at $N=60$ for the one-dimensional porous-medium inverse problem.}
\label{tab:porous1d-tuning}
\begin{tabular}{lrr}
\toprule
Method & Step size & Acceptance \\
\midrule
Exact gradient & $3.00\times10^{-5}$ & $0.637$ \\
Additive & $6.02\times10^{-4}$ & $0.677$ \\
Mismatched multiplicative & $1.81\times10^{-4}$ & $0.775$ \\
Matched Jackson & $2.00\times10^{-3}$ & $0.700$ \\
\bottomrule
\end{tabular}
\end{table}

\FloatBarrier
\subsection{Two-dimensional Darcy-flow inverse problem}
\label{subsec:darcy2d-experiment}

The multidimensional test uses
\begin{align}
-\nabla\cdot\left[k_N(s;\theta)\nabla p_N(s;\theta)\right]
&=1,
\qquad s\in D=(0,1)^2,
\label{eq:darcy2d-pde}
\\
p_N(s;\theta)&=0,
\qquad s\in\partial D,
\nonumber
\end{align}
with
\begin{equation}
k_N(s;\theta)
=
\exp\left\{
\frac{\rho}{2}
\left[
W_N\left(\theta_1(0.2+s_1)\right)
+
W_N\left(\theta_2(0.2+s_2)\right)
\right]
\right\}.
\label{eq:darcy2d-permeability}
\end{equation}
We take $a=0.8$, $b=1.4$, $\rho=0.18$, true parameter
$$
\theta^\dagger=(1.02,0.94),
$$
and a Gaussian prior centered at $(1,1)$ with covariance $0.12^2I$, restricted to the positive quadrant. Nine pressure observations on a $3\times3$ interior grid are generated at $N_{\mathrm{ref}}=24$ with independent Gaussian noise of standard deviation $0.003$.

The PDE is discretized on a $21\times21$ interior Cartesian grid using a conservative five-point flux scheme with arithmetic face conductivities. The exact gradient of the implemented negative log-likelihood uses one forward and one discrete adjoint solve. Each finite-scale coordinate quotient uses repeated evaluations of the same forward model. The coefficient and state fields in Figure~\ref{fig:darcy2d-fields} show the resolved spatial scales and the smoothing action of the elliptic forward map.

\begin{figure}[H]
\centering
\begin{subfigure}[t]{0.49\textwidth}
\centering
\includegraphics[width=\textwidth]{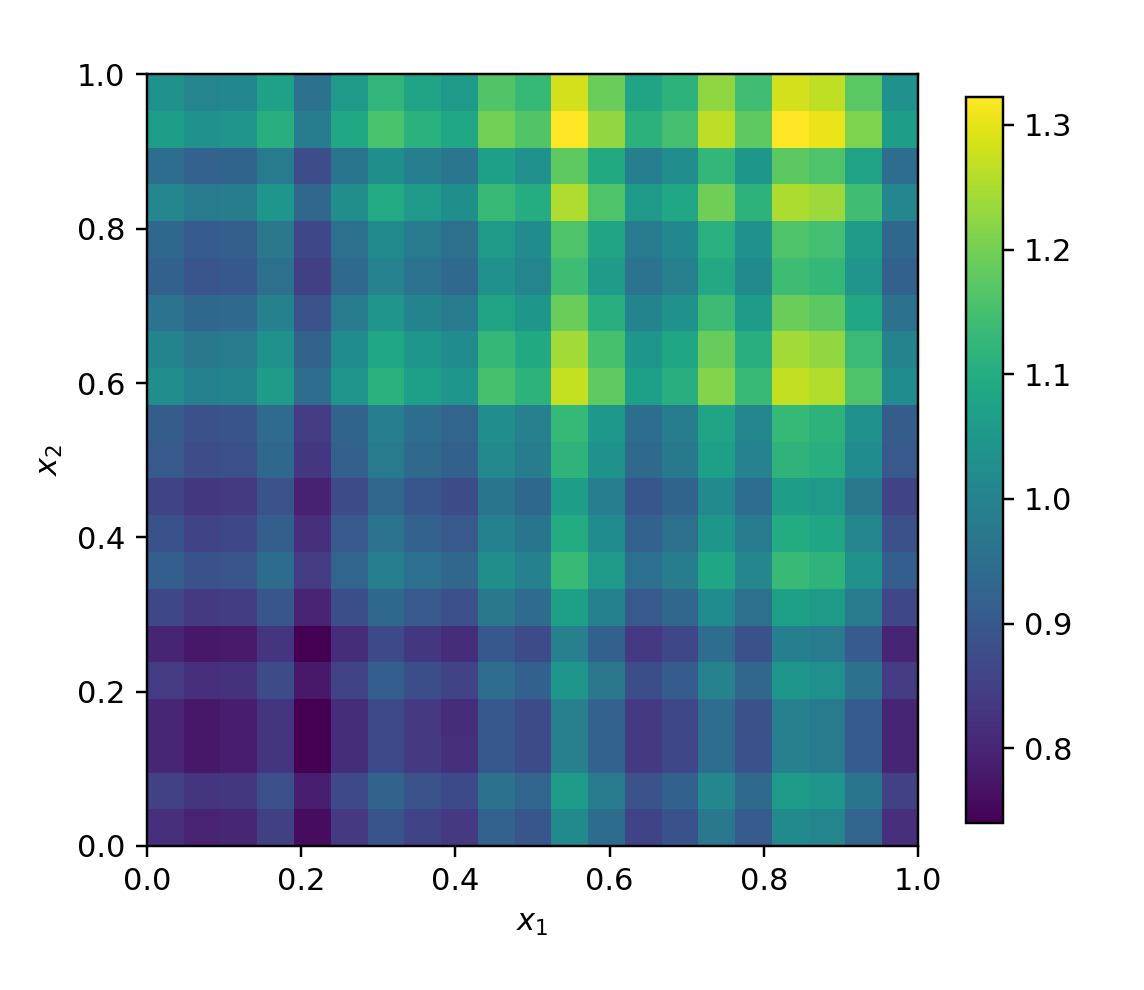}
\caption{Reference permeability.}
\end{subfigure}
\hfill
\begin{subfigure}[t]{0.49\textwidth}
\centering
\includegraphics[width=\textwidth]{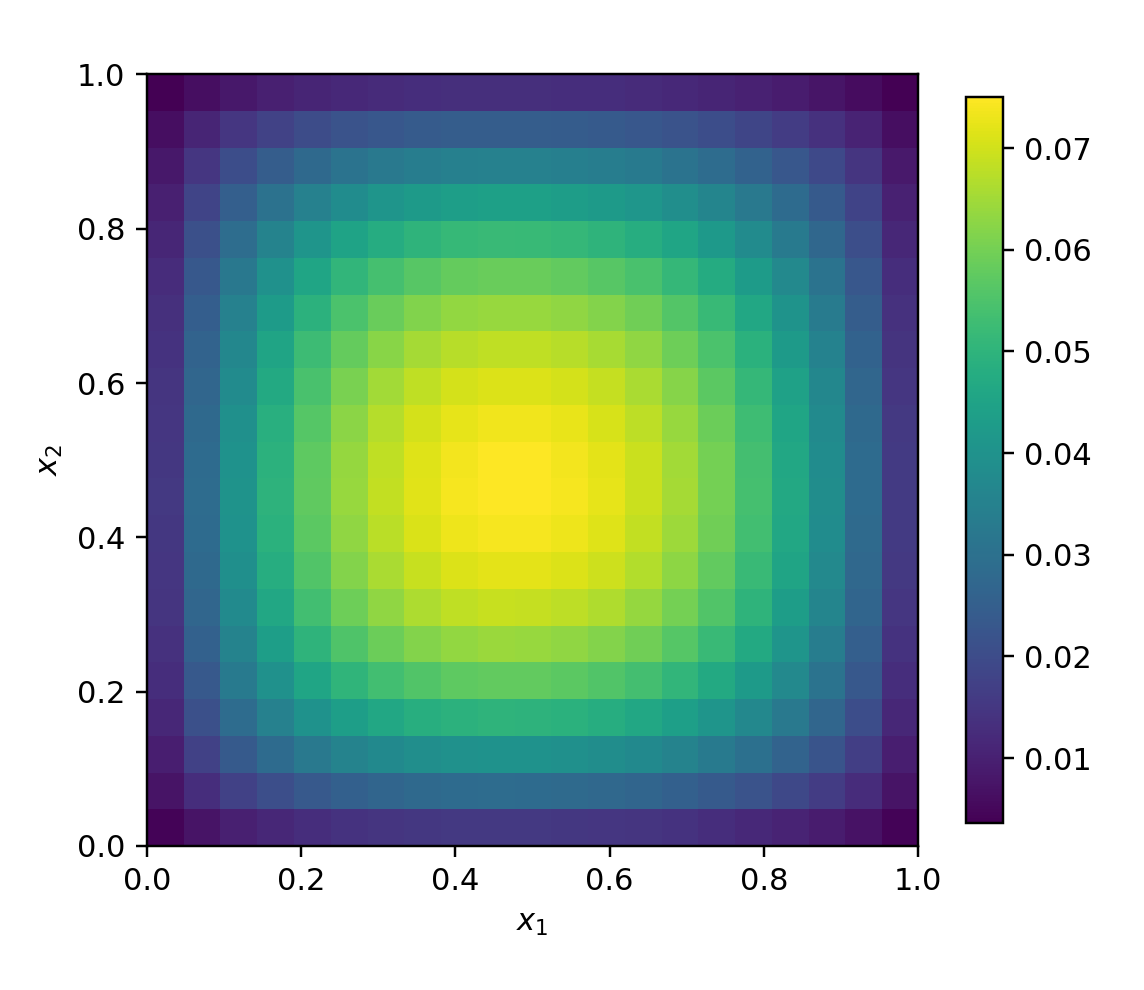}
\caption{Reference pressure.}
\end{subfigure}
\caption{Reference coefficient and state fields at $\theta^\dagger$ and $N_{\mathrm{ref}}=24$.}
\label{fig:darcy2d-fields}
\end{figure}

The main Darcy diagnostic is Figure~\ref{fig:darcy2d-resolution-acceptance}. Its left panel compares forward and gradient errors relative to $N_{\mathrm{ref}}=24$, both evaluated at the prior mean. The right panel reports acceptance after pilot step-size selection for the four proposal fields.

\begin{figure}[H]
\centering
\begin{subfigure}[t]{0.49\textwidth}
\centering
\includegraphics[width=\textwidth]{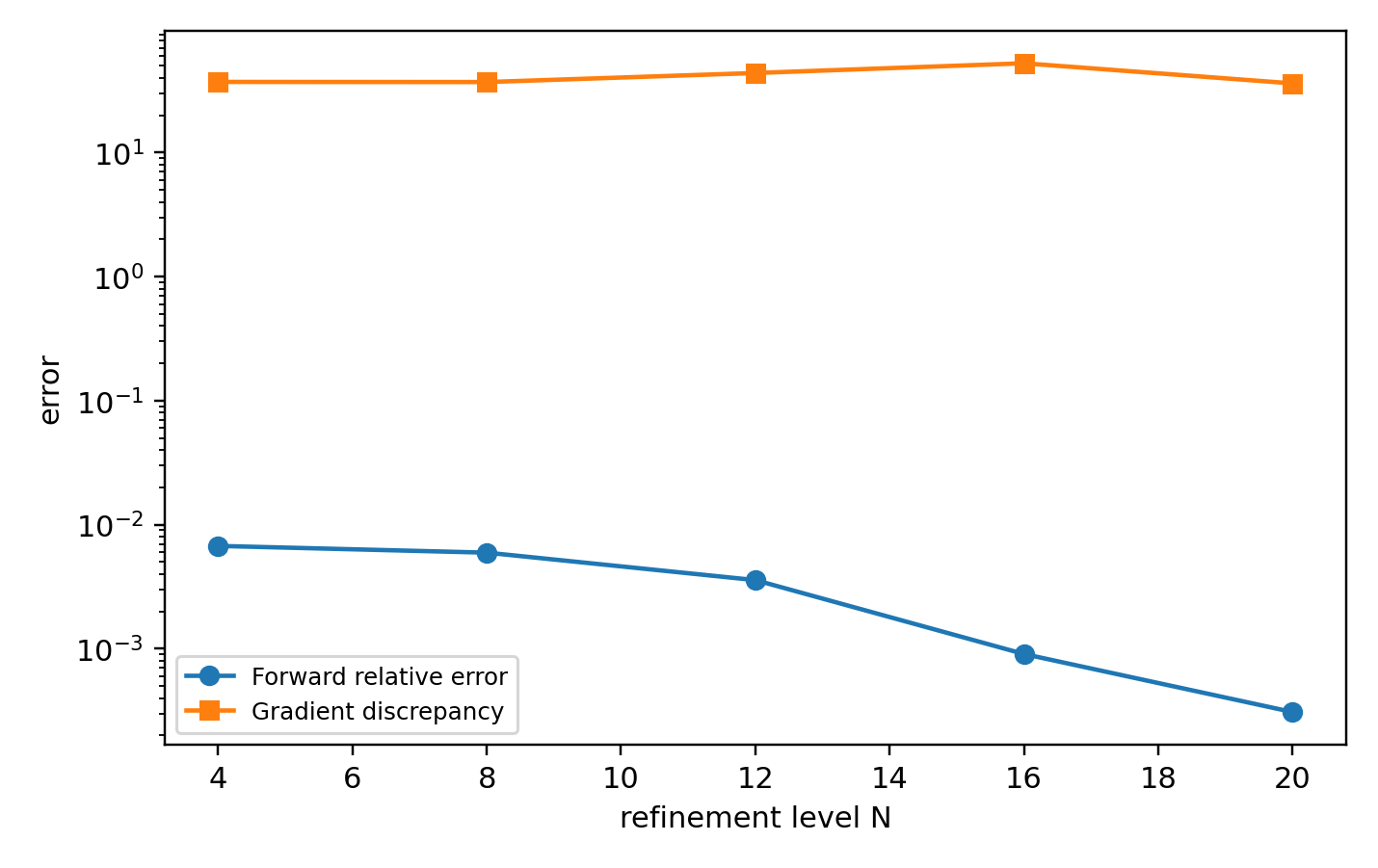}
\caption{Errors relative to $N_{\mathrm{ref}}=24$.}
\label{fig:darcy2d-separation}
\end{subfigure}
\hfill
\begin{subfigure}[t]{0.49\textwidth}
\centering
\includegraphics[width=\textwidth]{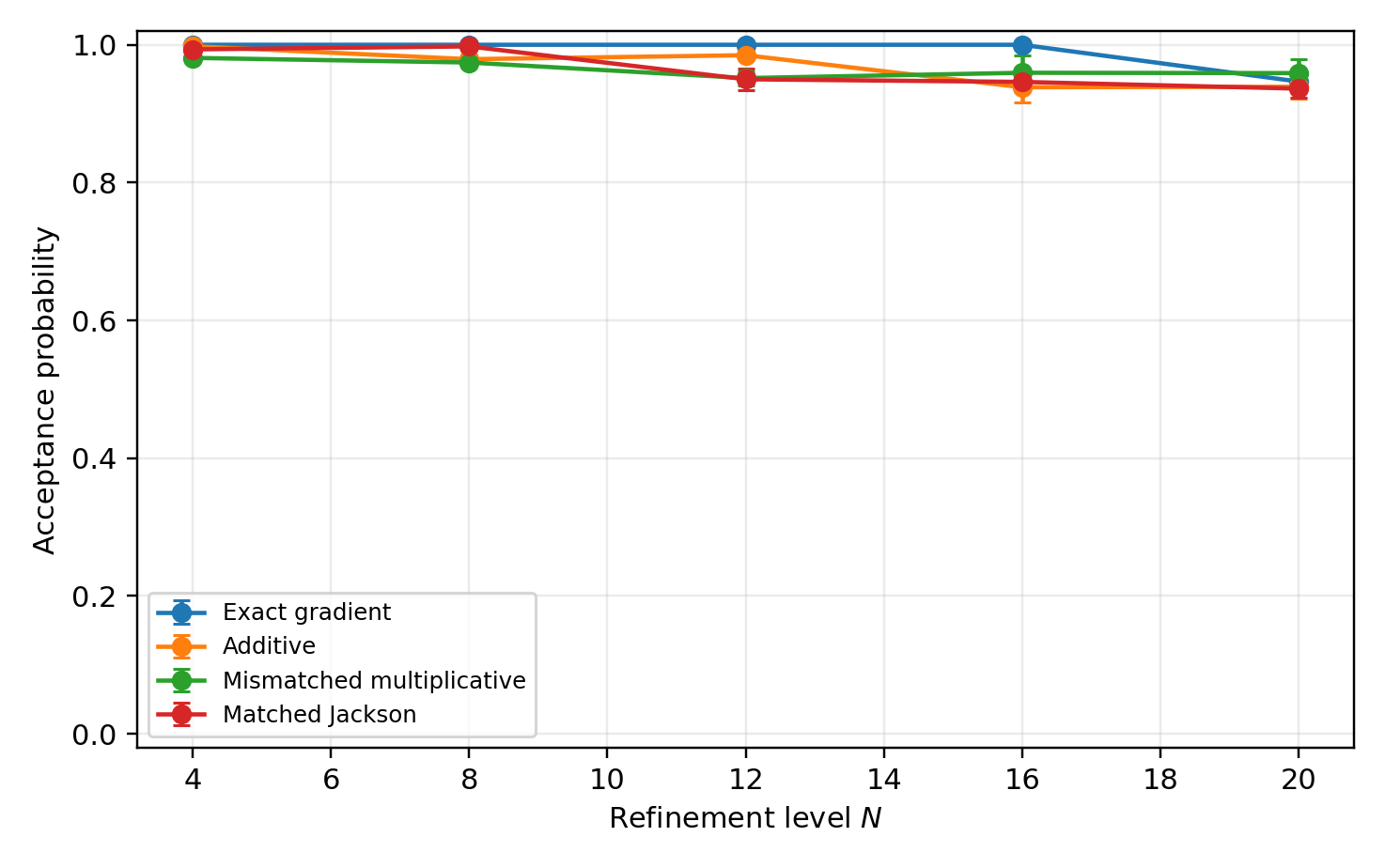}
\caption{Pilot-tuned acceptance.}
\label{fig:darcy2d-acceptance-panel}
\end{subfigure}
\caption{Resolution behavior in the Darcy problem. The forward error and posterior-gradient discrepancy are both evaluated at the prior mean; the former decreases steadily, whereas the latter remains of order $10^1$ over the tested levels. The right panel shows pilot-tuned acceptance.}
\label{fig:darcy2d-resolution-acceptance}
\end{figure}

As shown in Figure~\ref{fig:darcy2d-separation}, the forward relative error at the prior mean decreases from $6.74\times10^{-3}$ at $N=4$ to $3.08\times10^{-4}$ at $N=20$. In contrast, the norm of the exact-gradient difference relative to $N_{\mathrm{ref}}=24$ takes the values $37.2$, $37.1$, $43.9$, $52.6$, and $36.2$. The forward approximation therefore improves steadily while the sensitivity discrepancy does not. This extends the forward--sensitivity separation to a multidimensional PDE setting. 

\FloatBarrier
\subsection{Interpretation and limitations}

The experiments exhibit three refinement patterns. Figures~\ref{fig:scale-posterior-acceptance} and
\ref{fig:porous1d-posterior-acceptance}, together with
Tables~\ref{tab:scale-tuning} and \ref{tab:porous1d-tuning}, show posterior
stabilization accompanied by strong exact-gradient retuning in the two nested
Weierstrass-based inverse problems. The porous-medium experiment also shows
that the intrinsic dilation can support a substantially larger tuned step in
this model. Figure~\ref{fig:darcy2d-separation} shows that forward convergence
need not imply sensitivity convergence in a two-dimensional inverse problem.

The kernel analysis is local and concerns a fixed number of proposal steps on
admissible compact sets; it does not provide a resolution-uniform spectral-gap
or mixing-time estimate.

\section{Conclusion}
\label{sec:conclusion}

This paper isolates a refinement regime in which probability-level approximation and sensitivity-based computation have different limits. The mechanism is not specific to Weierstrass functions: the elementary wiggly family $r_\varepsilon\psi(\cdot/\varepsilon)$ gives the basic scaling, and the periodic elliptic problem in Section~\ref{subsec:pde-origin} shows that the same mismatch can be generated by a standard homogenization mechanism rather than inserted directly into the target.  The single-scale PDE corrector is small in the forward map and likelihood but retains an $O(1)$ exact sensitivity when the inverse parameter enters the fast variable. The nested Weierstrass family then models repeated corrector contributions across a geometric hierarchy, makes the obstruction more persistent by accumulating all resolved scales, sharpens the mismatch into the geometric contrast $a^N$ versus $(ab)^N$, and supplies an intrinsic multiplicative scale for the Jackson construction.

Metropolized finite-scale proposals separate proposal geometry from the unstable exact derivative. Measurable kick--drift--kick maps remain exact after Metropolis correction even when the proposal field is not the true gradient, and local field convergence propagates to fixed-length proposals, acceptance functions, and bounded-Lipschitz kernels. The matched Jackson quotient contributes additional structure through dilation covariance and exact closure of the canonical rough component at the intrinsic scale.

The numerical results concentrate on the scale-structured construction. The Weierstrass-based inverse problems show posterior stabilization together with resolution-dependent exact-gradient retuning, while the Darcy calculation demonstrates forward convergence without sensitivity convergence in two dimensions. Together with the analysis, these results show that a model-matched finite-scale field can remain refinement-consistent even when the exact infinitesimal force does not. The proposal and kernel limits are local and fixed-length rather than resolution-uniform long-time mixing results.

\appendix

\section{Uniform Hölder estimates for the Weierstrass family}
\label{app:holder}

This appendix derives the uniform Hölder estimate used in Section~\ref{subsec:holder-finite-scale} and in the compact proposal analysis. The main point is that the modulus is independent of the truncation level, so it remains valid in the rough limit.

Assume $\phi$ is one-periodic and Lipschitz with constant $L_\phi$. Let $\alpha\in(0,1]$ satisfy
$$
ab^\alpha<1.
$$

Set
$$
C_\phi
=
\max\left\{L_\phi,2\|\phi\|_{L^\infty}\right\}.
$$
For every $t\geq0$ and $\alpha\in(0,1]$,
$$
\min\left\{L_\phi t,2\|\phi\|_{L^\infty}\right\}
\leq
C_\phi t^\alpha.
$$
Indeed, $t\leq t^\alpha$ for $t\leq1$, while $1\leq t^\alpha$ for $t\geq1$. Periodicity and Lipschitz continuity therefore imply
\begin{align*}
\left|
\phi\left(b^nx\right)-\phi\left(b^ny\right)
\right|
&\leq
\min\left\{L_\phi b^n|x-y|,2\|\phi\|_{L^\infty}\right\}
\\
&\leq
C_\phi b^{\alpha n}|x-y|^\alpha.
\end{align*}
It follows that, uniformly in $N$,
\begin{align*}
|W_N(x)-W_N(y)|
&\leq
C_\phi|x-y|^\alpha
\sum_{n=0}^{N}(ab^\alpha)^n
\\
&\leq
\frac{C_\phi}{1-ab^\alpha}|x-y|^\alpha.
\end{align*}
Passing to the uniformly convergent limit gives the same estimate for $W$. Thus the family $\{W_N\}$ and its limit share a common Hölder modulus for every $\alpha\in(0,1]$ satisfying $ab^\alpha<1$.

On each compact interval away from zero, the closed formula equation~\eqref{eq:matched-closed-force} implies that $\J_bW$ inherits the same Hölder exponent, up to the smooth weight $x^{-1}$.

\end{document}